%% file: abparticles.tex
\documentclass[twoside,a4paper,12pt]{article}

\input{macros.tex}
\usepackage{bbm}

\begin{document}


\title{Recurrence and Density Decay  for \\ Diffusion-Limited Annihilating Systems}
\author{M. Cabezas, L. T. Rolla, V. Sidoravicius
\\ \small
Institute for Pure and Applied Mathematics, Rio de Janeiro
\\ \small
Department of Mathematics, Pontifical Catholic University of Chile
\\ \small
Argentinian National Research Council at the University of Buenos Aires
\\ \small
NYU-ECNU Institute of Mathematical Sciences at NYU Shanghai
\\ \small
Courant Institute of Mathematical Sciences, New York University
}

\date{}

\maketitle

\begin{abstract}
We study an infinite system of moving particles, where each particle is of type~$A$ or~$B$.
Particles perform independent random walks at rates $D_A > 0$ and $D_B \geq 0$, and the interaction is given by mutual annihilation $A+B \to \emptyset$.
The initial condition is i.i.d.~with finite first moment.
We show that this system is site-recurrent, that is, each site is visited infinitely many times.
We also generalize a lower bound on the density decay of Bramson and Lebowitz by considering a construction that handles different jump rates.
\end{abstract}

This preprint has the same numbering for sections, theorems, equations and figures
as the published article ``\emph{Probab. Theory Related Fields 170 (2018), 587-615}''

\input{introab.tex}
\input{construction.tex}
\input{zd.tex}
\input{sparsewalksab.tex}
\input{generalgraphsab.tex}

\appendix
\input{appendix.tex}

\section*{Acknowledgement}

We thank E.~Andjel for fruitful discussions. M.C. and V.S thank MSRI for hospitality and support. This project was supported by grants Programa Iniciativa Científica Milenio grant number NC130062 through Nucleus Millennium Stochastic Models of Complex and Disordered Systems, PIP 11220130100521CO, PICT-2015-3154, PICT-2013-2137, PICT-2012-2744, Conicet-45955 and MinCyT-BR-13/14.

\renewcommand{\baselinestretch}{1}
\parskip 0pt
\small

\bibliographystyle{abbrv}
\bibliography{bib/leo}

\end{document}

%% file: macros.tex
\usepackage[hidelinks]{hyperref}
\usepackage{amscd}
\usepackage{amsfonts}
\usepackage{amsmath}
\usepackage{amssymb}
\usepackage{amstext}
\usepackage{amsthm}
\usepackage{color}
\usepackage[nodate]{datetime}
\usepackage{dsfont}
\usepackage{fancyhdr}
\usepackage{graphicx}
\usepackage[latin1]{inputenc}
\usepackage{mathrsfs}
\usepackage{mathtools}
\usepackage{psfrag}
\usepackage{srcltx}
\usepackage[normalem]{ulem}
\usepackage{verbatim}

\usepackage[hmargin=30mm,top=30mm,bottom=30mm,paperwidth=210mm,paperheight=11in]{geometry}

\renewcommand{\baselinestretch}{1.1}
\usepackage{etoolbox}
\makeatletter
\patchcmd{\@addmarginpar}{\ifodd\c@page}{\ifodd\c@page\@tempcnta\m@ne}{}{}
\makeatother
\reversemarginpar

\DeclareMathOperator*{\sgn}{sgn}

\newtheorem{theorem}{Theorem}
\newtheorem{proposition}{Proposition}
\newtheorem{lemma}{Lemma}
\newtheorem*{lemma*}{Lemma}

\theoremstyle{remark}
\newtheorem{remark}{Remark}

\newcommand{\E}{\mathbb{E}}
\newcommand{\I}{\mathds{1}}
\newcommand{\N}{\mathbb{N}}
\newcommand{\Pb}{\mathbb{P}}
\newcommand{\R}{\mathbb{R}}
\newcommand{\Z}{\mathbb{Z}}

\newcommand{\mua}{\mu^{\scriptscriptstyle A}}
\newcommand{\mub}{\mu^{\scriptscriptstyle B}}
\newcommand{\avoid}{/\!\!/}
\newcommand{\meet}{\!{\not|}\;}

\newcommand{\A}{\mathcal{A}}
\newcommand{\B}{\mathcal{B}}
\newcommand{\C}{\mathcal{C}}

\newcommand{\T}{\mathcal{T}}
\newcommand{\X}{\mathcal{X}}
\newcommand{\Y}{\mathcal{Y}}

\renewcommand{\AA}{\mathscr{A}}
\newcommand{\BB}{\mathscr{B}}
\newcommand{\CC}{\mathscr{C}}

\newcommand{\oo}{{\boldsymbol{o}}}

\renewcommand{\leq}{\leqslant}

\renewcommand{\geq}{\geqslant}

\usepackage{wrapfig}
\usepackage{afterpage}

%% file: introab.tex
\section{Introduction}

In this paper we study an infinite system of moving particles, where particles can be of two types, $A$ or $B$.
Particles of opposite type mutually annihilate when they meet. Particles of type~$A$, or simply $A$-particles, jump at rate $D_A > 0$, and $B$-particles jump at rate $D_B \geq 0$. Several particles of the same type are allowed to share a site, and they do not interact among themselves. We consider the question of whether sites are visited infinitely often, and the related question of asymptotic decay of particle density.

Interest in long-time behavior of two-type annihilating particle systems, in particular with different jump rates, naturally stems from different areas of mathematics and physics.{\footnotemark}
These models have attracted much attention in the physics literature, especially after it was observed that some chemical reactions with two diffusing reactants exhibit anomalous kinetics in low dimensions.
More precisely, the evolution of the density of constituents depends strongly on the initial spatial fluctuations, and for dimensions $d<4$ its decay is slower than predicted by mean-field rate equations.
This was first noted in the seminal work~\cite{ovchinnikov-zeldovich-78}, and described in more detail in~\cite{toussaint-wilczek-83}.

\afterpage
{\footnotetext{%
Models with mutual annihilation were originally introduced in chemical physics for the study of radiation-chemical processes in polymers. It was proposed in~\cite{koritskii-etal-60} that a radical may move along a polymer chain, and that the act of recombination takes place when two migrating radicals encounter one another. A diffusive mechanism for the motion of radicals was proposed in~\cite{ovchinnikov-belyi-68}, and a kinetic equation describing concentration of free radicals as function of the time was derived. The same model served as a prototype of multi-type diffusion-limited chemical reactions with annihilation or inert compound outcome~\cite{balagurov-vaks-73,ovchinnikov-zeldovich-78}, and as caricature modeling particle-antiparticle annihilation of superheavy magnetic monopoles in the very early universe~\cite{toussaint-wilczek-83}.

Our motivation comes from the study of driven-dissipative lattice gases which undergo absorbing-state phase transitions. The authors arrived to the present model as a caricature of a system starting from an active state with critical density~\cite{dickman-rolla-sidoravicius-10,rolla-sidoravicius-12}. The $A$-particles should correspond to regions that are slightly supercritical due to fluctuations, whereas $B$-particles represent slightly subcritical regions. Surprisingly enough, some of the techniques developed in this paper have been applied with success in the study of the original model~\cite{cabezas-rolla-sidoravicius-14}.}}

Mathematically rigorous results came in a series of papers by Bramson and Lebowitz~\cite{bramson-lebowitz-88,bramson-lebowitz-90,bramson-lebowitz-91,bramson-lebowitz-91b,bramson-lebowitz-01}.
They obtained the asymptotic density decay for uniform nearest-neighbor walks on $\Z^d$ with jump rates $D_A=D_B=1$, and Poisson or Bernoulli i.i.d.~initial conditions.
For different initial densities $\mua_0 < \mub_0$, they showed that
\[
\mua_t \sim
\begin{cases}
e^{-c\, \sqrt{t}} ,& d=1,
\\
e^{-c\, t / \log t} ,& d=2,
\\
e^{-c\, t } ,& d \geq 3,
\end{cases}
\]
settling down conflicting predictions from theoretical physics.
For equal initial densities $\mua_0 = \mub_0$ they proved
\[
\mu_t \sim
\begin{cases}
t^{-d/4}, & d \leq 4,
\\
t^{-1}, & d > 4,
\end{cases}
\]
in agreement with heuristic arguments of spatial segregation for $d\leq 4$ and mean-field rate equations for $d>4$.
They also studied in detail the spatial structure of the system in low dimensions and obtained its hydrodynamic limit~\cite{bramson-lebowitz-01}.

The question of site recurrence for stochastic annihilating systems was first raised by Erd\H{o}s and Ney~\cite{erdos-ney-74}, and answered affirmatively for one-type systems in dimension one~\cite{adelman-76,lootgieter-77,schwartz-78}.
At the same time, additive and cancellative systems became one of the central topics in the field of interacting particle systems, and important progress was made in their understanding~\cite[etc]{holley-liggett-75,griffeath-78,holley-stroock-79,griffeath-79}.
The question of site recurrence for one-type annihilating random walks in arbitrary dimension was answered by Griffeath~\cite{griffeath-78} for a particular class of initial conditions, and the i.i.d.\ case was settled by Arratia~\cite{arratia-83}. 
Both approaches used an equivalence between one-type annihilating random walks and coalescing random walks or voter model sets with odd parity.

However, available methods and techniques did not encompass the case of two-type systems.
In fact, as observed in~\cite{bramson-lebowitz-91}, the analysis of the two-type particle annihilating process is considerably more difficult due to the lack of comparison with an attractive particle system.
Another important mathematical challenge emerges when~$A$ and $B$-particles jump at different rates, causing most existing approaches for two-type systems to break down.

In this paper we tackle the question of site recurrence for this model, and in the course of the proof we also obtain a universal lower bound on density decay.
Below is a brief description of our results.

Our main theorem states that, almost surely, every site is visited infinitely often.
The jump rates~$D_A$ and~$D_B$ need not be equal (though one of them must be positive).
The underlying space where the system is defined can be any graph~$G$ baring a group of automorphisms~$\Gamma$ such that, for every $x,y\in G$, there is~$\pi\in\Gamma$ for which $\pi x=y$ and $\pi y=x$.
These graphs are called \emph{generously transitive graphs}, see Section~\ref{sec:generalgraphs}.
It is also assumed that, for some group~$\Gamma$ that makes~$G$ generously transitive, $p(\pi x,\pi y)=p(x,y)$ for all $\pi\in\Gamma$ and $x,y\in G$.
No assumptions are made on the tail of the jump distribution~$p(\oo,z)$ as $z\to\infty$.
The initial condition is assumed to be an i.i.d.~field with finite first moment.
Moreover, the lower bound $\mu_t \geq \frac{c}{t}$ proved in~\cite{bramson-lebowitz-91} is extended to the same level of generality.

Theorems will be stated and proved progressively, and the last result encompasses the previous ones.
Each proof introduces an extra layer of difficulty and requires new ideas.
Formal statements, as well as comments on the methods and ideas of the proofs, appear at the beginning of each section.

In Section~\ref{sec:construction} we present a graphical description of the process.
This particular construction will be used throughout the rest of the paper.
We also state properties of mass conservation and monotonicity.

In Section~\ref{sec:samerates} we assume that $G=\Z^d$, $p(x,x+y)=p(\oo,y)=p(\oo,-y)$, and that the jump rates are $D_A = D_B = 1$.
We first revisit the proof of $\mu_t \geq \frac{c}{t}$ from~\cite{bramson-lebowitz-91}, and then push the argument in order to obtain site recurrence.

In Section~\ref{sec:danedb}, we prove $\mu_t \geq \frac{c}{t}$ and site recurrence for the general case $0 \leq D_B < D_A$.
The previous construction is replaced by a new one, which is suitable to handle different jump rates.
This way the argument from Section~\ref{sec:samerates} can be generalized to this setting.
Site recurrence for $D_B=0$ is shown separately, via a re-sampling technique.

In Section~\ref{sec:generalgraphs} we extend the previous arguments to generously transitive graphs.
These include uniform nearest-neighbor walks on $\Z^d$, regular trees, Cayley graphs, as well as products thereof.

The proofs translate without significant modifications to multi-type systems, as well as to one-type systems on generously transitive graphs.{\footnotemark}

{\footnotetext{In multi-type systems particles are of types $A_1,A_2,\dots,A_M$, and jump at rate~$1$ according to a generously transitive transition kernel~$p(\cdot,\cdot)$.
Interaction is given by $A_i+A_j \to \emptyset$ for any $i \ne j$.
Each site initially contains one particle of type $A_i$ with probability $\frac{p}{M}$ and no particles with probability $1-p$, independently of other sites.
In the one-type system the interaction is given by $A+A\to\emptyset$ and the initial condition is i.i.d.\ Bernoulli.
The proofs given in the next sections for site recurrence work in these settings.
For the one-type system, there is a simpler and more general proof in~\cite{BenjaminiFoxallGurel-GurevichJungeKesten16}.}}

Finally, proofs of well-definedness of the graphical construction, mass conservation, monotonicity, ergodic properties of random walks, and $0\,$-$1$ laws, are postponed to Appendix~\ref{sec:appendix}.

%% file: construction.tex
\section{Graphical construction}
\label{sec:construction}

In this section we give an explicit construction of the system described informally in the Introduction.
This construction will be used in the rest of the paper in order to prove properties thereof.

Let $D_A \geq 0$ and $D_B \geq 0$ denote the jump rates.
For simplicity we consider the graph~$\Z^d$ and a jump distribution $p:\Z^d \times \Z^d \to [0,1]$ satisfying $p(x,x+y)=p(\oo,y)$ for all~$x$ and $y$, where~$\oo$ denotes the origin.
What follows still holds true for any graph having a transitive unimodular group of automorphisms under which the transition kernel $p(\cdot,\cdot)$ is diagonally invariant.

The evolution will be denoted $\xi=(\xi_t)_{t\geq 0}$, where $\xi_t\in\Z^{\Z^d}$ for $t\geq 0$, such that $\xi_t(x)=k$ means that~$k$ particles of type~$A$ are present at site~$x$ at time~$t$, and $\xi_t(x)=-k$ means that~$k$ particles of type~$B$ are present.
We assume that $(\xi_0(x))_{x\in\Z^d}$ is i.i.d.~with marginal~$\nu$, where $\nu$ is a given distribution on~$\Z$ with finite first moment.

We denote by $\xi^{xy}$ the configuration obtained from $\xi$ after an $A$-particle jumps from $x$ to $y$ or a $B$-particle jumps from $y$ to $x$, which is given by
\[
\xi^{xy}(z) =
\begin{cases}
\xi(x) - 1, & z=x,
\\
\xi(y) + 1, & z=y,
\\
\xi(z), & \text{otherwise}
.
\end{cases}
\]
The formal generator is given by
\[
\mathcal{L}f(\xi) = \sum_{x,y} D_A [\xi(x)]^+ p(x,y) \big(f(\xi^{xy})-f(\xi)\big) + D_B [\xi(x)]^- p(x,y)  \big(f(\xi^{yx})-f(\xi)\big)
.
\]

Below we describe an explicit construction of this process, where the number of particles per site and the putative trajectory of each particle prior to annihilation are sampled beforehand.
Later on we will show that such construction is well-defined.

From general results in~\cite{andjel-82}, there exists a unique process $(\xi_t)_{t\geq 0}$ corresponding to the above generator.
Moreover, this is a Feller process with respect to a topology that is weak enough so that the probability of any local event\footnote{An event is ``local'' if its occurrence is determined by $\big(\xi_t(x)\big)_{|x|\leq M,t\leq M}$ for some $M<\infty$.} can be approximated by taking a system that starts without particles outside a large enough box.
Therefore, any construction given by the limit of finite systems whose particles interact according to the previous description will yield a process with the same distribution.
We insist on using \emph{this particular construction} which handles infinitely many particles simultaneously, because it allows the use of re-sampling in Section~\ref{sec:fixedobstacles}, and has good spatial ergodicity properties that lead to a simple proof of Lemma~\ref{lem:mtp} below with its numerous consequences, including $0\,$-$1$ Laws.

So let us describe the construction.
Each $A$-particle is identified by a label $(x,j)$ for $1\leq j \leq \xi_0(x)$, and each $B$-particle by a label $(x,j)$ for $\xi_0(x) \leq j \leq -1$.
For each $x\in \Z^d$ and $j\in \Z^*=\Z\setminus\{0\}$, let $S^{x,j}=(S_t^{x,j})_{t\geq 0}$ be a continuous-time random walk starting at $x$ which jumps according to the transition kernel $p(\cdot,\cdot)$, and whose jump rate is $D_A$ for $j>0$ and $D_B$ for $j<0$, independent over~$x$ and~$j$.
Moreover, let~$h^{x,j}$ be independent, uniform on $[0,1]$.
We call $(\xi_0(x))_{x\in\Z^d}$ the \emph{initial condition}, and refer to the pair $(S,h) = \left( (S^{x,j})_{x\in \Z^d,j\in \Z^*},(h^{x,j})_{x\in \Z^d,j\in\Z^*} \right)$ as the \emph{instructions}.
These fields are sampled independently.

To each particle $(x,j)$ we assign a \emph{putative trajectory} $S^{x,j}$ and a \emph{braveness}~$h^{x,j}$.
Particles follow their putative trajectory as time evolves, until they are annihilated.
When a particle jumps on a site occupied by particles of the opposite type, it mutually annihilates with the bravest one.

Let $M(x,j,x',j',z,t)=M(x',j',x,j,z,t)$ denote the event that particle $(x,j)$ and $(x',j')$ are present in $\xi_0$ and that they mutually annihilate at site $z$ during $[0, t]$.
We need to show that this construction is well-defined.
That is, we need to show that, almost surely, for each $x$, $j$, $x'$, $j'$, $z$, and $t$, when determining whether or not $M(x,j,x',j',z,t)$ occurs, it is possible to decide its occurrence from the initial condition and instructions.

\begin{lemma}
\label{lem:existence}
For any distribution of the initial condition~$\xi_0$ satisfying
\[ \sup_{x \in \Z^d} \E|\xi_0(x)| < \infty , \]
the above construction is a.s.\ well-defined and translation covariant.
\end{lemma}
\begin{proof}
Postponed to Appendix~\ref{sec:appendix}.
\end{proof}

Let \[M(x,j,t) = \bigcup_{x',z\in\Z^d,j'\in\Z^*}M(x,j,x',j',z,t)\] denote the event that particle $(x,j)$ has been annihilated by time $t$, and let
\[
V(x,j,t)=[1 \leq j \leq \xi_0(x) \text{ or } \xi_0(x) \leq j \leq -1] \setminus M(x,j,t)
\]
denote the event that particle $(x,j)$ is alive at time~$t$.
For $x\in\Z^d$, write $\delta_{x}$ for the field in $\Z^{\Z^d}$ given by $\delta_x(y)=1$ for $y=x$ and $0$ for $y\ne x$.
So $\delta_x$ and $-\delta_x$ denote respectively the configuration having a single $A$ or $B$ particle, located at site $x$.
For an event $A$, let $\I_A$ denote the corresponding indicator function.
With this notation, the configuration $\xi_t$ at any time~$t$ is given by
\[
\xi_t = \sum_{x\in\Z^d} \sum_{\sigma = \pm 1} \sum_{j \in \N} \sigma \cdot \I_{V(x,\sigma j,t)} \cdot \delta_{S^{x,\sigma j}_t}
.
\]
Let $\T_{x,j}$ denote the time of annihilation of particle~$(x,j)$, with $\T_{x,j}=0$ in case $(x,j)$ is not present on $\xi_0$.
With this definition, $\T_{x,j}>t$ if and only if $V(x,j,t)$ occurs.

We finish this section with the following facts.

\begin{lemma}
[Mass conservation]
\label{lem:mtp}
Let
\[
\mua_t = \E[\xi_t(\oo)^+]
, \qquad
\mub_t = \E[\xi_t(\oo)^-]
, \qquad
\rho_t = \mua_t + \mub_t = \E|\xi_t(\oo)|
\]
denote the density of $A$-particles, $B$-particles and total density of particles per site at time~$t$.
Let
\[
\Theta_t = \tfrac{1}{2} \sum_{x,x',j,j'} \Pb\big[{M(x,j,x',j',\oo,t)}\big]
\]
denote the density of annihilations per site by time~$t$.
Then~$\Theta_t$ is increasing in~$t$ and
\[
\mua_t - \mua_0 = \mub_t - \mub_0 = \frac{\rho_t - \rho_0}{2} = - \Theta_t.
\]
Therefore, $\mua_t-\mub_t$ is constant in time.
Moreover,
\[
\mua_t = \sum_{j\in\N} \Pb\big(V(\oo,j,t)\big)
\]
and, in particular, $\Pb\big[ \T_{\oo,1}<\infty \big] = 1$ if an only if $\mua_t \to 0$ as $t\to\infty$.
Analogously for $-j$ instead of $j$ and $B$ instead of $A$.
\end{lemma}
\begin{proof}
Postponed to Appendix~\ref{sec:appendix}.
\end{proof}

\begin{lemma}
[Monotonicity]
\label{lem:monotononicity}
Suppose that $\xi_0(x) \leq \xi'_0(x)$ for all $x\in\Z^d$.
Let $\xi$ and $\xi'$ be a pair of two-type annihilating systems constructed using the same instructions.
Then
\[
\T_{x,j} \leq \T'_{x,j}
\text{ and } 
\T'_{x,-j} \leq \T_{x,-j}
\text{ for all }
x \in \Z^d
\text{ and }
j \in \N
.
\]
In particular,
\[
\xi_t(x) \leq \xi'_t(x) \text{ for all } x\in\Z^d \text{ and } t\geq0.
\]
\end{lemma}
\begin{proof}
Postponed to Appendix~\ref{sec:appendix}.
\end{proof}

%% file: zd.tex
\section{Site recurrence and density decay on the lattice}
\label{sec:samerates}

In this section we prove the following:
\begin{theorem}
\label{thm:zd}
Let $(\xi_t)_{t \geq 0}$ be a two-type annihilating particle system on~$\Z^d$.
Suppose that the initial condition $\xi_0\in\Z^{\Z^d}$ is an i.i.d.~field whose marginal~$\nu$ on~$\Z$ is non-degenerate and has finite first moment.
Suppose that the jump rates are $D_A=D_B=1$ and that the jump distribution~$p(\cdot,\cdot)$ satisfies $p(x,x+y)=p(\oo,y)=p(\oo,-y)$.
Then~$\xi$ is site recurrent, i.e.,
\(
\Pb\left[ \xi_t(\oo) = 0 \text{ eventually} \, \right] = 0
.
\)
\end{theorem}

We start constructing a copy~$\xi^m$ of the system~$\xi$, which has the same distribution as~$\xi$ and differs from it by a set of particles having small density~$m$.
This coupling is used to obtain a lower bound for the particle density decay, as in~\cite{bramson-lebowitz-91}.
We then extend these arguments to handle additivity and correlations, finally proving site recurrence.

\subsection{Coupled evolutions and tracers}

We start with a coupling of initial conditions.
From now on we assume that $\nu(0)>0$ and $\nu(1)>0$, and in Section~\ref{sec:generalic} we consider general~$\nu$.

\begin{lemma}
\label{lem:couplingic}
For every $m$ small enough depending on~$\nu$, there exists a coupling $\big(\xi_{0}(x),\xi^m_{0}(x)\big)_{x\in \Z^d}$ such that both $(\xi^m_0(x))_{x\in \Z^d}$ and $({\xi}_0(x))_{x\in \Z^d}$ are i.i.d.~fields with marginal~$\nu$ and such that $\xi_{0}(x)-\xi^m_{0}(x)= \pm 1$ with probability $\frac{m}{2}$ each, and $\xi_{0}(x)=\xi^m_{0}(x)$ with probability $1-m$, independently over $x\in \Z^d$.
\end{lemma}
\begin{proof}
Let $({\xi}_0(x))_{x\in \Z^d}$ be i.i.d.~with distribution~$\nu$, and let $(U^x)_{x\in \Z^d}$ be i.i.d.~$U[0,1]$ and independent of $(\xi_{0}(x))_{x\in \Z^d}$.
Write
\[
\B^1_r:=\{x\in \Z^d: \xi_{0}(x)=0; \ U^x \leq r\}
\]
and
\[
\B^2_r:=\{x\in \Z^d: \xi_{0}(x)=1; \ U^x \leq r\}.
\]
Take $r_1$ such that $\Pb[\oo \in \B^1_{r_1}]=\frac{m}{2}$ and $r_2$ such that $\Pb[\oo \in \B^2_{r_2}]=\frac{m}{2}$.
Define
\[
\xi^m_{0}(x):=
\begin{cases}
1,         & x\in \B^1_{r_1},
\\
0,         & x\in \B^2_{r_2},
\\
\xi_{0}(x) & \textrm{ otherwise.}
\end{cases}
\]
Then the pair $\big(\xi_{0}(x),\xi^m_{0}(x)\big)_{x\in \Z^d}$ has the desired properties.
\end{proof}

We define two systems, $\xi=(\xi_t)_{t\geq0}$ and $\xi^m=(\xi^m_t)_{t\geq0}$, using the same instructions but initial conditions~$\xi_0$ and~$\xi^m_0$ given by the above lemma.
Let $\A_m^+:=\{x:\xi^m_0(x)-\xi_0(x)=+1\}$, $\A_m^-:=\{x:\xi^m_0(x)-\xi_0(x)=-1\}$, and $\A_m:= \A_m^+ \cup \A_m^-$.
In the sequel we define a family of tracers which will keep track of the discrepancies between $\xi$ and $\xi^m$.

As a warm up, suppose that $\A_m = \A_m^+ = \{ x \}$ consists of a single site $x$, and $A_m^- = \emptyset$.
Imagine for instance $\xi_0(x) \geq 0$ and $\xi^m_0(x)=\xi_0(x)+1$.
We define the \emph{tracer} $X^{x}=(X^{x}_t)_{t\geq 0}$ by following the difference between~$\xi$ and~$\xi^m$ due to the presence of this extra particle.
At $t=0$ we have $X^{x}_0=x$.
Initially, $X^{x}$ will follow the trajectory of the extra $A$-particle on $\xi^m$, until it is annihilated by the collision with a $B$-particle.
After this time, the difference between~$\xi$ and~$\xi^m$ will persist, but it will be transferred to another particle.

The particle being tracked by tracer~$X^x$ may be annihilated in~$\xi^m$, and this can happen in two ways.
First case: the $B$-particle responsible for annihilation in~$\xi^m$ remains alive in~$\xi$.
This happens if the annihilation is due to the tracked $A$-particle jumping on a site occupied by $B$-particles, or a $B$-particle jumping on a site occupied by the tracked $A$-particle alone.
Second case: the same $B$-particle which annihilates with the tracked $A$-particle in~$\xi^m$ annihilates simultaneously with some other $A$-particle in~$\xi$.
This happens if the annihilation is due to a $B$-particle jumping on a site occupied by several $A$-particles including the tracked one.
In both cases, the difference between~$\xi^m$ and~$\xi$ immediately after annihilation of the tracked $A$-particle is now due to another particle's presence: an extra $B$-particle present at~$\xi$, or an extra $A$-particle present at $\xi^m$, depending on the case.
From this instant onwards, the tracer~$X^x$ will follow this extra particle.

This procedure can be continued indefinitely.
This difference will last for all times, and we obtain $(X^{x}_t)_{t\geq0}$ with the property that
\begin{equation}
\label{eq:diffonetracer}
\xi^m_t - \xi_t = \delta_{X^x_t}
\end{equation}
for all $t\geq 0$.
Notice that~$(X^x_t)_{t\geq 0}$ is distributed as a random walk with jump rate $D_A=D_B=1$ and jump distribution~$p(\cdot,\cdot)$.

The tracer described above always corresponds to an extra $A$-particle in $\xi^m$ or an extra $B$-particle in~$\xi$.
We call this tracer a $\oplus$-tracer.
Analogously, if we had assumed that $\A_m = \A_m^- = \{ x \}$, $\A_m^+ = \emptyset$, $\xi^m_0(x) \geq 0$, and $\xi_0(x) = \xi_0^m(x)+1$, we would end up with a tracer that for all times corresponds to an extra $A$-particle in $\xi$ or an extra $B$-particle in~$\xi^m$.
Such a tracer is called a $\ominus$-tracer.

The next step is to define the set of tracers $\{X^{x}\}_{x\in{\mathcal A}_m}$ in the case where ${\mathcal A}_m$ is not a singleton.
The same construction is still well-defined, but it may happen that a given tracer corresponds to a discrepancy only for a finite period of time.
It occurs when two tracers of opposite sign correspond respectively to two extra particles of opposite types, both present in the same system and absent in the other, and these particles mutually annihilate in that system.
In this case the discrepancies disappear, and both tracers are left with nothing to track.
Another scenario is when two tracers of opposite sign correspond respectively to two extra particles of the same type, one of them present in~$\xi$ and the other in~$\xi^m$, and they are simultaneously annihilated by some particle of opposite type present in both systems.
Again, in this case both tracers are left with nothing to track.
These are the only possible cases.

For convenience, once this happens to a given tracer, we extend its trajectory to all times, by sampling a random walk independent of anything else.
We say that a tracer is \emph{active} before such event, and \emph{wandering} after that.
Also, for convenience, at $t=0$ we start with one wandering tracer at sites~$x \not\in \A_m$.
This way we get a set $\{X^x\}_{x\in\Z^d}$ whose distribution is that of a set of independent walks.
Observe that each active tracer remains active at least until the first time when it meets an active tracer of opposite sign.

Analogously to the case of a single tracer, we have
\begin{equation}
\label{eq:diffmanytracers}
\xi^m_t-\xi_t=\sum_{\sigma=\pm} \ \sum_{y\in\A_m^\sigma} \sigma \, \I_{[X^y \text{ active at time }t]} \cdot \delta_{X^y_t}
.
\end{equation}
Finally, observe that the presence of an active tracer at site $w$ at time $t$ implies that $\xi_t(w)\ne 0$ or $\xi^m_t(w)\ne 0$.

\subsection{Lower bound for density evolution}

In this section we present an argument for the study of density decay that works in any dimension.

\begin{theorem}
[\cite{bramson-lebowitz-91}]
\label{thm:bl}
Under the assumptions of Theorem~\ref{thm:zd},
\[
\rho_t \geq \frac{c}{t}
\]
for all large enough $t$, where $c>0$ is a universal constant.
\end{theorem}

We now prove the above result.
The essence of the argument is taken from~\cite{bramson-lebowitz-91}.

Assume that $\nu(0)>0$ and $\nu(1)>0$.
In Section~\ref{sec:generalic} we consider general~$\nu$.
For $t>0$, we write $X^z \meet X^y$ if $X^z_s=X^y_s$ for some $s\in[0,t]$, and $X^z \avoid X^y$ otherwise.
Analogously for $X^z \meet w$ and $X^z \avoid w$.
Let $(Y^w_s)_{s\geq 0}$ denote a random walk with jump rate $2$ starting at~$w$.
The main estimate is that, for any $y\in\Z^d$,
\begin{align}
\Pb
\big[ X^z \meet X^y
\text{ for some } z\in\A_m \setminus y
\big]
&
\leq
\,
\sum_{\mathclap{z\in\Z^d \setminus y}}
\,
\Pb \left[ z\in\A_m, \, X^z \meet X^y
\big.
\right]
\nonumber
\\
&
=
m
\,
\sum_{\mathclap{z\in\Z^d \setminus y}}
\,
\Pb
\left[ X^z - X^y \meet \oo
\big.
\right]
\label{eq:differenceoftracerequalrates}
\\
&
=
m
\sum_{\mathclap{z\in\Z^d \setminus y}}
\Pb
\left[ Y^{z-y} \meet \oo \right]
\nonumber
\\
&
=
m
\sum_{\mathclap{z\in\Z^d \setminus y}}
\Pb
\left[ Y^\oo \meet y-z \right]
\label{eq:switchtricktwopaths}
\\
&
=
m
\,
\E
\left[\# \text{ new sites visited by }Y^{\oo} \right]
\nonumber
\\
&
\leq
2 m t
\nonumber
.
\end{align}

Fix some $\delta^2<\frac{1}{2}$, and for $t>0$ let
\begin{equation*}
m = \frac{\delta^2}{t}
.
\end{equation*}
Take $t$ large enough such that $m$ is sufficiently small for Lemma~\ref{lem:couplingic} to hold.
Plugging this into the previous estimate gives
\begin{equation}
\label{eq:avoideachother}
\Pb
\big[ X^z \avoid X^y \text{ for all } z\in\A_m \setminus y \big]
\geq
1-2\delta^2
.
\end{equation}

Now notice that
\[
\Pb\left[ \xi_t(\oo) \ne 0 \text{ or } \xi^m_t(\oo) \ne 0 \right]
\leq
\Pb\left[ \xi_t(\oo) \ne 0 \right]
+
\Pb\left[ \xi^m_t(\oo) \ne 0 \right]
=
2 \cdot \Pb\left[ \xi_t(\oo) \ne 0 \right]
\leq
2 \rho_t
,
\]
and thus
\begin{align}
2 \rho_t
& \geq
\Pb\left[ X^y_t = \oo \text{ for some } y \in \A_m \text{ and }X^y \text{ active at time } t \right]
\nonumber
\\
& \geq
\Pb\left[ X^y_t = \oo \text{ for some } y \in \A_m \text{ and } X^z \avoid X^y \text{ for all } z\in\A_m \setminus y \right]
\nonumber
\\
& =
\Pb\left[ X^y_t = \oo \text{ for unique } y \in \A_m \text{ and } X^z \avoid X^y \text{ for all } z\in\A_m \setminus y \right]
\nonumber
\\
& =
\textstyle \sum_y \Pb\left[ y \in \A_m,\ X^y_t = \oo, \text{ and } X^z \avoid X^y \text{ for all } z\in\A_m \setminus y \right]
\nonumber
\\
& =
\textstyle \sum_y \Pb\left[ \oo \in \A_m,\ X^\oo_t = y, \text{ and } X^z \avoid X^\oo \text{ for all } z\in\A_m \setminus \oo \right]
\nonumber
\\
& =
\Pb\left[ \oo \in \A_m \text{ and } X^z \avoid X^\oo \text{ for all } z \in \A_m \setminus \oo \right]
\label{eq:activedensity}
,
\end{align}
where the first two equalities hold because $X^z \avoid X^y$ implies that $X^z_t \ne \oo$.

Plugging~(\ref{eq:avoideachother}) we get for $t\geq \delta^2$
\begin{equation*}
2\rho_t
\geq
m\, (1-2\delta^2) = \frac{\delta^2(1-2\delta^2)}{t}
,
\end{equation*}
proving the theorem with $c=\delta^2(\frac{1}{2}-\delta^2)$.

The above argument gives $c=\frac{1}{16}$ for any dimension, any jump distribution and any integrable initial distribution.
It can be improved to $c=\frac{1}{8}$ by using the fact that $\oplus$-tracers can only be canceled by $\ominus$-tracers and vice-versa.
A mean-field heuristics, which is supposed to be the worst case, gives $c=1$.

\subsection{Site recurrence}
\label{ss:siterecurrencezd}

In the proof of Theorem~\ref{thm:bl} one gets a lower bound $\frac{c}{t}$ for the probability of finding an active tracer at $\oo$ \emph{at time}~$t$.
We want to extend that argument and obtain a \emph{constant} lower bound for the probability of finding an active tracer at~$\oo$ \emph{at any time before}~$t$.
This, together with the following $0\,$-$1$ law, will imply the main result.

\begin{lemma}
\label{lem:zeroone}
$
\Pb\left[ \, \xi_t(\oo) = 0 \text{ eventually} \, \right] = 0 \text{ or } 1.
$
\end{lemma}
\begin{proof}
Postponed to Appendix~\ref{sec:appendix}.
\end{proof}

We now prove Theorem~\ref{thm:zd} in the case when the transition kernel~$p(\cdot,\cdot)$ yields a transient random walk on~$\Z^d$.

The core of the argument is to obtain a lower bound for the probability of the event
\begin{equation}\label{eq:BB}
\BB_t := \big[ \text{an active tracer visits } \oo \text{ during } [0,t] \big]
.
\end{equation}

We use the following decomposition:
\begin{align}
\Pb(\BB_t)
&
\geqslant
\Pb \big[ \text{for some } y\in \A_m,\, X^y \meet \oo \text{ and } X^z \avoid X^y \text{ for all } z\in\A_m\setminus y \big]
\nonumber
\\
&
\geq
\Pb \big[ X^y \meet \oo \text{ for a unique } y\in \A_m, \text{ and }  X^z \avoid X^y \text{ for all } z\in\A_m\setminus y \big]
\nonumber
\\
&
=
m
\textstyle
\sum_y
\Pb
\big[ X^y \meet \oo ,\, X^z \avoid \oo \text{ and } X^z \avoid X^y \text{ for all } z\in\A_m \setminus y \big]
\label{eq:activesumunique}
.
\end{align}

By~(\ref{eq:avoideachother}) we have $\Pb\big[X^z \meet X^y \text{ for some } z \in\A_m \setminus y \big]\leq 2\delta^2$, and an analogous argument gives
\begin{multline*}
\Pb
\big[X^z \meet \oo \text{ for some } z \in\A_m \setminus y \big]
\leq \\ \leq
m \, \E
\left[ \,\text{number of sites visited by $X^\oo$ up to time $t$}\, \right]
\leq
m (t+1) \leq 2 \delta^2
\end{multline*}
for $t$ large enough.

When summing over~$y$, we cannot afford loosing the additive factor of $2\delta^2$ for each~$y \in \Z^d$.
We would like to deal with a multiplicative factor instead.
However, in the occurrence of \emph{both} events $[X^y \meet \oo]$ and $[X^y \avoid X^z]$ there is a negative correlation which we cannot handle.
To work around this issue, we introduce the following event.
We say that the path $X^y=(X^y_t)_{t \geq 0}$ is \emph{good} if
\begin{equation*}
\Pb\left[ \, X^z \avoid X^y \text{ for all } z\in\A_m \setminus y \,\Big|\, X^y \, \right] \geq 1-\delta.
\end{equation*}
The key observation is that~(\ref{eq:avoideachother}) yields
\begin{equation}
\nonumber
\Pb\left[ X^y \text{ good} \right] \geq 1-2\delta
,
\end{equation}
as
$
\delta
\cdot
\Pb\left[ X^y \text{ bad\,} \right]
\leq
\Pb\left[ X^z \meet X^y \text{ for some }z\in \A_m\setminus y \big. \,\middle|\, X^y \text{ bad\,} \right]
\cdot
\Pb\left[ X^y \text{ bad\,} \right]
\leq
2\delta^2
.
$
Now
\begin{align}
\Pb(\BB_t)
\geq
& \,
m
\textstyle
\sum_y
\Pb
\big[ X^y \text{ good},\, X^y \meet \oo ,\, X^z \avoid \oo \text{ and } X^z \avoid X^y \text{ for all } z\in\A_m \setminus y \big]
\nonumber
\\
=
& \,
m
\textstyle
\sum_y
\Pb
\big[ X^y \meet \oo ,\, X^y \text{ good} \big]
\times
\nonumber
\\
& \,
\E
\left[ \Pb\left[ X^z \avoid \oo \text{ and } X^z \avoid X^y \text{ for all } z\in\A_m\setminus y \,\big| \, X^y \right] \,\Big|\, X^y \meet \oo ,\, X^y \text{ good}  \right]
\nonumber
.
\end{align}
The last identity follows from the fact that both $[X^y \meet \oo]$ and $[X^y \text{ good}]$ are $X^y$-measurable.
Notice that the conditional probability is only being integrated on a subset of $[X^y \text{ good}]$, and thus by a simple union bound and translation invariance we get
\begin{align}
\Pb(\BB_t)
&
\geq
m(1-\delta-2\delta^2)
\textstyle
\sum_y
\Pb
\big[ X^y \meet \oo ,\, X^y \text{ good} \big]
\nonumber
\\
&
=
m(1-\delta-2\delta^2)
\textstyle
\sum_y
\Pb
\big[ X^\oo \meet y ,\, X^\oo \text{ good} \big]
\label{eq:switchtrickonepath}
\\
&
=
m(1-\delta-2\delta^2)
\
\E
\big[ \# \text{ sites visited by }X^\oo \text{ during } [0,t],\, X^\oo \text{ good} \big]
\label{eq:activeexpectedrange}
.
\end{align}
Let $R_t$ denote the number of sites visited by a random walk during $[0,t]$.
Since the walk is transient, writing $\gamma = \Pb[X^\oo \text{ never returns to }\oo]$ we have that $\frac{R_t}{t} \to \gamma$ in $L^1$.
Hence, for~$t$ large enough
\begin{align}
\Pb(\BB_t)
&
\geq
m (1-\delta-2\delta^2)
\ \gamma\ t\ (1-2\delta-\delta)
=
\gamma\, \delta^2\, (1-\delta-2\delta^2)
\, (1-3\delta)
\label{eq:visitpositive}
.
\end{align}
Choosing $\delta=\frac{1}{4}$ and writing $\epsilon=\frac{\gamma}{103}$, the above estimate yields $\Pb(\BB_t) \geq \epsilon$ for all~$t$ large enough.

Finally, let us check that this uniform lower bound implies site recurrence.
For~$T>0$ fixed, consider the decomposition
\[
\BB_t = \BB_t^{T-} \cup \BB_t^{T+},
\]
where $\BB_t^{T_+}$ and $\BB_t^{T_-}$ denote the events that an active tracer visits~$\oo$ respectively during $[T,t]$ and during $[0,T]$.
Now notice that
\[
\Pb(\BB_t^{T-}) \leq m \E[R_T] \leq \frac{\delta^2 T}{t} \to 0 \quad \text{as } t \to \infty,
\]
and therefore
\[
\epsilon \leq \liminf_{t\to\infty} \Pb(\BB_t^{T+}) \leq 2 \Pb[\xi_s(\oo)\ne 0 \text{ for some } s\geq T] \underset{T\to\infty}{\longrightarrow} 2\Pb[\oo \text{ visited i.o.\ in }\xi]
.
\]
With Lemma~\ref{lem:zeroone} this finishes the proof for the transient case.

Finally, suppose that the transition kernel $p(x,y)_{x,y\in \Z^d}$ yields a recurrent walk on~$\Z^d$.
We consider a coupling $(\xi_0,\xi^m_0)$ such that $\xi^m_0(x)=\xi_0(x)$ for all $x \ne \oo$ and $\xi^m_0(\oo)$ is sampled independently of $\xi_0(\oo)$.
We define the systems~$\xi$ and~$\xi^m$ using the same evolution rules.
Note that there is a positive probability that the initial conditions~$\xi_0$ and~$\xi^m_0$ differ at $\oo$ by a single particle.
On the occurrence of this event, the difference between~$\xi$ and~$\xi^m$ will evolve according to a tracer $X^\oo$.
By assumption, the tracer is recurrent, and hence $\xi(\oo)\neq\xi^m(\oo)$ infinitely often.
By the $0\,$-$1$ law it follows that, almost surely, $\oo$ is visited infinitely often in~$\xi$.

\subsection{General initial distribution}
\label{sec:generalic}

We now drop the condition that $\nu(0)>0$ and $\nu(1)>0$.
By assumption, the initial distribution~$\nu$ is non-degenerate, so there exist $n\in\Z$ and $K\in\N$ such that $\nu(n)>0$ and $\nu(n+K)>0$.
In the sequel we indicate what changes in the previous proofs suffice to accommodate this case.

The coupling of initial conditions is given as follows.
Let
\[
\B^1_r:=\{x\in \Z^d: \xi_{0}(x)=n; \ U^x \leq r\}
,\
\B^2_r:=\{x\in \Z^d: \xi_{0}(x)=n+K; \ U^x \leq r\}.
\]
As before, we take $r_1$ such that $\Pb[\oo \in \B^1_{r_1}]=\frac{m}{2}$ and $r_2$ such that $\Pb[\oo \in \B^2_{r_2}]=\frac{m}{2}$, and define
\[
\xi^m_{0}(x):=
\begin{cases}
n + K,        & x\in \B^1_{r_1},
\\
n,            & x\in \B^2_{r_2},
\\
\xi_{0}(x),   & \textrm{otherwise.}
\end{cases}
\]
Let $\A_m^+ = \B^1_{r_1}$, $\A_m^- = \B^2_{r_2}$, and $\A_m = \A_m^+ \cup \A_m^-$.
Then
\begin{equation*}
\xi^m_0-\xi_0=\sum_{\sigma=\pm} \ \sum_{y\in\A_m^\sigma} \sigma \, K \cdot \delta_y
.
\end{equation*}
At each site $y \in \A_m^+$ there are $K$ different $\oplus$-tracers, which we denote by $X^{y,j}$, $j=1,\dots,K$.
Analogously, at each $y \in \A_m^-$ there are $K$ different $\ominus$-tracers, also denoted by $X^{y,j}$, $j=1,\dots,K$.
As before we say that a tracer is \emph{active} until the time when it finds no particle left to track.
Thus, for $t \geq 0$,
\begin{equation*}
\xi^m_t-\xi_t=\sum_{\sigma=\pm} \ \sum_{y\in\A_m^\sigma} \ \sum_{j=1}^{K} \sigma \, \I_{[X^{y,j} \text{ active at time }t]} \cdot \delta_{X^{y,j}_t}
\end{equation*}
and the presence of an active tracer at site $w$ at time $t$ implies that $\xi_t(w)\ne 0$ or $\xi^m_t(w)\ne 0$.
Recall that each tracer may only become wandering when it meets a tracer of opposite sign.
In particular it stays active at least until the first time when it meets a tracer which started at a different site.

We now prove the lower bound for density decay under general initial conditions.
Following the proof of Theorem~\ref{thm:bl}, a union bound on~$j$ gives
\begin{multline}
\Pb
\big[ X^{z,j} \meet X^{y,1}
\text{ for some } z\in\A_m \setminus y \text{ and } j
\big]
\leq
\\
K \cdot
\Pb
\big[ X^{z,1} \meet X^{y,1}
\text{ for some } z\in\A_m \setminus y
\big]
\leq
2 K m t
=
2 \delta^2
\nonumber
\end{multline}
if we take $m = \frac{\delta^2}{Kt}$.
Also, following the arguments of~(\ref{eq:activedensity}) we get
\begin{align*}
2 \rho_t
& \geq
\textstyle \sum_y \Pb\left[ y \in \A_m,\ X^{y,1}_t = \oo, \text{ and } X^{z,j} \avoid X^{y,1} \text{ for all } z\in\A_m \setminus y \text{ and } j \right]
\\
& =
\Pb\left[ \oo \in \A_m \text{ and } X^{z,j} \avoid X^{\oo,1} \text{ for all } z \in \A_m \setminus \oo \text{ and } j \right]
,
\end{align*}
finishing the proof of Theorem~\ref{thm:bl} with $c=\delta^2(\frac{1}{2}-\delta^2)K^{-1}$.
The term $K^{-1}$ can be eliminated, as we do after~(\ref{eq:estimatewithK}).

For Theorem~\ref{thm:zd}, repeating the arguments between~(\ref{eq:activesumunique}),~(\ref{eq:activeexpectedrange}), and~(\ref{eq:visitpositive}) gives
\begin{align*}
\Pb(\BB_t)
&
\geq
m
\textstyle
\sum_y
\Pb
\big[ X^{y,1} \meet \oo ,\, X^{z,j} \avoid \oo \text{ and } X^{z,j} \avoid X^{y,1} \text{ for all } z\in\A_m \setminus y \text{ and } j \big]
\\
&
\geq
m(1-\delta-2\delta^2)
\
\E
\big[ \# \text{ sites visited by }X^{\oo,1} \text{ during } [0,t],\, X^{\oo,1} \text{ good} \big]
\\
&
\geq
m (1-\delta-2\delta^2)
\ \gamma\ t\ (1-2\delta-\delta)
=
\gamma\, \delta^2\, (1-\delta-2\delta^2)
\, (1-3\delta)
\, K^{-1} > 0
.
\end{align*}
From this estimate the previous proof may be concluded along the same lines.

%% file: sparsewalksab.tex
\section{Distinct jump rates}
\label{sec:danedb}

In this section we prove the following:
\begin{theorem}
\label{thm:DAneDB}
Let $(\xi_t)_{t \geq 0}$ be a two-type annihilating particle system on~$\Z^d$.
Suppose that the initial condition $\xi_0\in\Z^{\Z^d}$ is an i.i.d.~field whose marginal~$\nu$ on~$\Z$ is non-degenerate and has finite first moment.
Suppose that the jump rates are $0 \leq D_B \leq D_A = 1$ and that the jump distribution~$p(\cdot,\cdot)$ satisfies $p(x,x+y)=p(\oo,y)=p(\oo,-y)$.
Then there exists a universal $c>0$ such that
\[
\rho_t \geq \frac{c }{t}
\]
for all large enough $t$.
Moreover, unless $D_B = 0$ and $\mub_0>\mua_0$,
\[
\Pb\left[ \big. \xi_t(\oo) \text{ eventually constant} \right] = 0,
\]
i.e., the system $\xi$ is site recurrent.
\end{theorem}

We would like to implement the strategy of tracking differences between coupled evolutions, as in Section~\ref{sec:samerates}.
The challenge is to construct a coupling which saves the previous strategy from breaking down.
We want to do this by unraveling different tracers from their dependence, while keeping some control on their trajectories.

\subsection{Coupling evolutions via tracers}

We are going to construct a coupling between two systems so that the trajectories of the tracers are given a priori.
We present this coupling first for $D_A = D_B$ and then for $D_A \neq D_B$.

Assume that $D_A=D_B=1$, and that the initial conditions $\xi_0$ and $\xi^m_0$ differ by only one particle.
The pair $(\xi,\xi^m)$ constructed in Section~\ref{sec:samerates} satisfies~(\ref{eq:diffonetracer}) and yields $(X^x_t)_{t \geq 0}$, distributed as a random walk with jump rate~$1$.
Let us consider an alternative construction.
Instead of obtaining~$X^x$ by following the tracked particle until its annihilation, we sample~$X^x$ beforehand, and force the tracked particle to follow~$X^x$ until its annihilation.
Now the tracer is actually \emph{dragging} the particle, but to keep consistency with previous sections we still say that it is \emph{tracking} it.
When the particle being tracked is annihilated, there is another particle which will carry the difference between systems~$\xi$ and~$\xi^m$.
From that time on, this particle forgets its instructions and follows the trajectory of the tracer (i.e., the tracer begins to track/drag that new particle).

This alternative construction yields a coupled pair of systems $(\xi,\xi^m)$ with the same distribution and satisfying~(\ref{eq:diffonetracer}).
Again, this is a $\oplus$-tracer, and for each time~$t$ it corresponds to an extra $A$-particle in~$\xi^m_t$ or an extra $B$-particle in~$\xi_t$.
The opposite holds for a $\ominus$-tracer.

We now consider the case $D_A = D_B $ but when~$\xi_0$ and~$\xi^m_0$ differ by infinitely many particles, as given by Lemma~\ref{lem:couplingic}.
Again, this is well-defined, except that the presence of two tracers of opposite sign at the same site may result in annihilations that leave both of them with no particles to track.
As before, we say that the tracer is active before that time and wandering after that.
This way we construct a pair $(\xi,\xi^m)$ satisfying~(\ref{eq:diffmanytracers}) and differing at $t=0$ on a random set~$\A_m$, which is Bernoulli and has intensity~$m$.

Now suppose $D_A \geq D_B \geq 0$.
We want a construction of $X^y$ that retains a good control on the set of sites visited by this path, which was used in the proof of site-recurrence for equal jump rates, particularly in~(\ref{eq:switchtrickonepath}).

Let~$\X^{y,j}$ denote a triple $\left( W^{y,j}, \T^A_{y,j} , \T^B_{y,j} \right)$, where $W^{y,j}=(W^{y,j}_n)_{n=0,1,2,\dots}$ is a discrete-time random walk with transition kernel $p(\cdot,\cdot)$ starting at $y$, $\T^B_{y,j} \subseteq \R_+$ is a Poisson clock with intensity~$D_B$, independent of $W^{y,j}$, and $\T^A_{y,j} \subseteq \R_+$ is obtained by adding an independent Poisson clock of intensity $D_A-D_B$ to $\T^B_{y,j}$.
The triples $\X^{y,j}$ are sampled independently over~$y$ and~$j$.

For each $y\in\A_m$ and each $j=1,\dots,K$, the tracer $X^{y,j}$ is defined by following the steps prescribed by $W^{y,j}$, and at each instant listening either to clock~$\T^A_{y,j}$ or~$\T^B_{y,j}$, depending on whether it is tracking an $A$-particle or a $B$-particle.
More precisely, each time~$t \in \T^A_{y,j}$ when a new Poissonian mark is found, $X^{y,j}$ moves to the next position in the sequence~$(W^{y,j}_n)_{n}$ if the particle being tracked it is of type~$A$, or if it is of type~$B$ and $t \in \T^B_{y,j}$.
If the particle being tracked is of type~$B$ and $t \in \T^A_{y,j} \setminus T^B_{y,j}$, nothing happens.

The rest of the construction is analogous to that described in Section~\ref{sec:samerates}.
Particles which are not being tracked move according to their respective instructions, which are the same in systems~$\xi$ and~$\xi^m$.
When a tracked particle is annihilated, the tracer starts to track another particle, unless there is a tracer of opposite sign at the same site and the differences cancel, in which case both tracers may be left with no particles to track.
We say that the tracer is active before that time, and wandering after that.
For completeness, after the tracer is no longer active it becomes wandering, and it listens to one of the clocks, say~$\T^A_{y,j}$.
Also for completeness, at sites $y \not \in \A_m$, we launch~$K$ tracers, which are wandering for all $t \geq 0$.

\begin{remark}
\label{rmk:couplingworks}
With the above construction, every active $\oplus$-tracer corresponds to an extra $A$-particle in~$\xi^m$ or an extra $B$-particle in~$\xi$, and the opposite holds for $\ominus$-tracers.
The difference between $\xi^m$ and $\xi$ is given by~(\ref{eq:diffmanytracers}).
An active tracer remains active at least until the first time when it meets a tracer of opposite sign.
Finally, the presence of an active tracer at site~$x$ at time~$t$ implies that $\xi_t(x) \ne 0$ or $\xi^m_t(x) \ne 0$.
\end{remark}

In order to obtain lower and upper bounds that will play the role of identity~(\ref{eq:switchtrickonepath}), we define the paths~$X^{y,j}_\pm = \big( X^{y,j}_\pm(t) \big)_{t\geq 0}$ by following the discrete path $W^{y,j}$ while listening respectively to clocks~$\T^A_{y,j}$ and~$\T^B_{y,j}$.
This way we have
\[ \left[ X^{y,j}_- \meet w \right] \subseteq \left[ X^{y,j} \meet w \right] \subseteq \left[ X^{y,j}_+ \meet w \right]  \]
for all $w\in G$.
Moreover, $X^{y,j}_\pm$ is distributed as a random walk with transition kernel $p(\cdot,\cdot)$ and jump rate given respectively by~$D_A$ and~$D_B$.
Also, since the pairs $(X^{y,j}_+,X^{y,j}_-)$ depend only on the respective~$\X^{y,j}$, they are independent over~$y$ and~$j$, and independent of $\A_m$.

We denote by~$\Pb$ the underlying probability measure.
This construction is well-defined and yields a coupled pair of systems~$(\xi,\xi^m)$ having the desired distribution.
The proof is in the same spirit as that of Lemma~\ref{lem:existence}, and will be omitted.

We will need yet another coupling.
Recall that~(\ref{eq:switchtricktwopaths}) was based on the fact that
\[ X^{z,1}-X^{y,1} \ \overset{\mathrm{d}}{=} \ Y^{z-y}, \]
where $Y^w$ is distributed as a random walk of jump rate $2 D_A = 2 D_B = 2$ starting at~$w$.
The following construction intends to provide $X^{z,1}$ and $X^{y,1}$ simultaneously, and still allow some control on the set of sites visited by $Y^{yz} := X^{z,1}-X^{y,1}$.

Let~$y$ and~$z$ be fixed.
For all $(x,j) \in \Z^d \times \{1,\dots,K\}$ except~$(y,1)$ and~$(z,1)$, we build the tracer~$X^{x,j}$ using~$\X^{x,j}$ as above.
The tracers~$X^{y,1}$ and $X^{z,1}$ will be \emph{entangled}, and they are constructed from the quintuple $\Y^{yz} = (Z^{yz},\T^A_y,\T^B_y,\T^A_z,\T^B_z)$ as follows.

Let $Z^{yz}=(Z^{yz}_n)_{n\in\N_0}$ be a discrete-time random walk that starts at $Z^{yz}_0 = z-y$ and jumps according to~$p(\cdot,\cdot)$.
At all times, the tracer~$X^{y,1}$ is listening either to clock~$\T^A_{y}$ or~$\T^B_{y}$, depending on the state of the particle it is tracking.
Analogously, the tracer~$X^{z,1}$ is listening either to clock~$\T^A_{z}$ or~$\T^B_{z}$.
When a new Poissonian mark is found at $t \in \T^A_{z}$, the tracer $X^{z,1}$ performs the next jump found in the sequence~$Z^{yz}$, i.e., it jumps by $Z^{yz}_{k+1}-Z^{yz}_k$ if the particle being tracked is of type~$A$ or if it is of type~$B$ and $t\in\T^B_z$.
If the particle being tracked is of type~$B$ and $t \in \T^A_{z} \setminus \T^B_{z}$, nothing happens.
Similarly, when a new Poissonian mark is found at $t \in \T^A_{y}$, the tracer $X^{y,1}$ performs the \emph{opposite} of the next jump found in the sequence~$Z^{yz}$, i.e., it jumps by $Z^{yz}_{k}-Z^{yz}_{k+1}$ if the particle being tracked is of type~$A$ or if it is of type~$B$ and $t\in\T^B_y$.
If the particle being tracked is of type~$B$ and $t \in \T^A_{y} \setminus \T^B_{y}$, nothing happens.
This way the difference~$Y^{yz}$ is reproducing in continuous time the discrete path prescribed by~$Z^{yz}$.

As before, each of $X^{y,1}$ and $X^{z,1}$ is active if it is actually tracking a particle, or else it is wandering, in which case it listens respectively to clock~$\T^A_{y}$ or $\T^A_{z}$.

We define the path~$Y^{yz}_+ = \big( Y^{yz}_+(t) \big)_{t \geq 0}$ by following the discrete path~$Z^{yz}$ and jumping on $\T^A_{y} \cup \T^A_{z}$, that is, it jumps when any of the clocks ring.
This way we have
\[ \left[ Y^{yz} \meet w \right] \subseteq \left[ Y^{yz}_+ \meet w \right] . \]
Moreover, $Y^{yz}_+$ is distributed as a random walk with transition kernel~$p(\cdot,\cdot)$ and jump rate $2$.
To indicate the use of this construction we denote the underlying probability measure by~$\Pb^{yz}$.
Notice that Remark~\ref{rmk:couplingworks} still holds for this construction.

\subsection{Density decay}

For brevity we introduce the notation $\tilde{\Pb}=\Pb(\,\cdot \,|\, y\in\A_m)$, $\hat{\Pb}=\Pb(\,\cdot\,|\,y\in\A_m,\,z\in\A_m)$, and $\hat{\Pb}^{yz}=\Pb^{yz}(\,\cdot\,|\,y\in\A_m,\,z\in\A_m)$.
Let $Y^{\boldsymbol o}_+$ be a continuous-time random walk with jump rate $2$ started at the origin.
For any $y\in\Z^d$, we have
\begin{align}
\tilde{\Pb}
\hspace{30mm}&\hspace{-30mm}
\big[ X^{z,j} \meet X^{y,1}
\text{ for some } z\in\A_m \setminus y \text{ and } j
\big]
\leq
\nonumber
\\&
\leq
mK
\,
\sum_{\mathclap{z\in\Z^d \setminus y}}
\,
\hat{\Pb}
\left[ X^{z,1} - X^{y,1} \meet \oo
\big.
\right]
\label{eq:differenceofwalks}
\\
&
=
mK
\,
\sum_{\mathclap{z\in\Z^d \setminus y}}
\,
\hat{\Pb}^{yz}
\left[ Y^{yz} \meet \oo
\big.
\right]
\nonumber
\\
&
\leq
mK
\sum_{\mathclap{z\in\Z^d \setminus y}}
\hat{\Pb}^{yz}
\left[ Y^{yz}_+ \meet \oo \right]
\nonumber
\\
&
=
mK
\sum_{\mathclap{z\in\Z^d \setminus y}}
\Pb
\left[ Y^\oo_+ \meet y-z \right]
\nonumber
\\
&
=
mK
\,
\E
\left[\text{\,number of new sites visited by }Y^{\oo}_+ \text{ up to time }t\,\right]
\nonumber
\\
&
\leq
2 \, m \, K \, t
\nonumber
.
\end{align}

Fix some $\delta^2<\frac{1}{2}$, and for each large enough $t$ let
\begin{equation*}
m = \frac{\delta^2}{Kt}
.
\end{equation*}
Plugging this into the previous estimate gives
\begin{equation}
\label{eq:avoideachotherDAneDB}
\tilde{\Pb}
\big[ X^{z,j} \avoid X^{y,1} \text{ for all } z\in\A_m \setminus y \text{ and } j \big]
\geq
1-2\delta^2
.
\end{equation}
Again as before,
\begin{align*}
2 \rho_t
\geq
\Pb\left[ \oo \in \A_m \text{ and } X^{z,j} \avoid X^{\oo,1} \text{ for all } z \in \A_m \setminus \oo \text{ and } j \right]
.
\end{align*}

Plugging~(\ref{eq:avoideachotherDAneDB}) we get for $t$ large enough
\begin{equation}
\label{eq:estimatewithK}
2\rho_t
\geq
m\, (1-2\delta^2) = \frac{\delta^2(1-2\delta^2)}{Kt}
,
\end{equation}
which would prove the theorem with $c=\delta^2(\frac{1}{2}-\delta^2)K^{-1}$.

Finally, let us eliminate the $K^{-1}$ term.
By symmetry,
\[
2 \rho_t = \E\left[ \big. |\xi_t(\oo)|+|\xi^m_t(\oo)| \right] \geq \E|\xi^m_t(\oo) - \xi_t(\oo)| = 2 \E[\xi^m_t(\oo)-\xi_t(\oo)]^+
.
\]
Hence
\begin{multline*}
\rho_t\geq \E\big[\,\xi^m_t(\oo)-\xi_t(\oo)\,\big]^+ = \\ = 
\E \big[ \big.
\text{number of active $\oplus$-tracers minus active $\ominus$-tracers at } \oo \text{ at time }t
\big]^+
.
\end{multline*}
Now observe that,
if at time $t$ the origin has a $\oplus$-tracer that has not yet met any tracer with different starting point, then this $\oplus$-tracer must be active and there cannot be a $\ominus$-tracer at $\oo$ at time $t$.
Therefore, continuing from the above lower bound we get
\begin{equation}
\nonumber
\begin{aligned}
\rho_t
& \geq
\E \left[ \# \left\{ (x,i) : y \in \A_m^+, X^{y,i}_t = \oo, X^{z,j} \avoid X^{y,i} \text{ for all } z \in \A_m\setminus y \text{ and } j \right\} \right]
\\
& =
\sum_{y \in \Z^d} \sum_{i=1}^K
\Pb \left[ y \in \A_m^+, X^{y,i}_t = \oo, X^{z,j} \avoid X^{y,i} \text{ for all } z \in \A_m\setminus y \text{ and } j \right]
\\
& =
\tfrac{m}{2} K
\sum_{y \in \Z^d}
\Pb \left[ X^{y,1}_t = \oo, X^{z,j} \avoid X^{y,1} \text{ for all } z \in \A_m\setminus y \text{ and } j \big\vert y\in\A_m^+ \right]
\\
& =
\tfrac{m}{2} K
\sum_{y \in \Z^d}
\Pb \left[ X^{\oo,1}_t = y, X^{z,j} \avoid X^{\oo,1} \text{ for all } z \in \A_m\setminus \oo \text{ and } j \big\vert \oo\in\A_m^+\right]
\\
& =
\tfrac{m}{2} K
\, \Pb \left[ X^{z,j} \avoid X^{\oo,1} \text{ for all } z \in \A_m\setminus \oo \text{ and } j \big\vert \oo\in\A_m^+\right]
\\ &
\geq
\tfrac{m}{2} K
\, (1-2\delta^2)
,
\end{aligned}
\end{equation}
proving the lower bound with $c=\delta^2(\frac{1}{2}-\delta^2)$.

\subsection{Site recurrence}

We now prove site recurrence.
We assume that $D_B > 0$.
The case $D_B = 0$ is considered in Section~\ref{sec:fixedobstacles}.

As in Section~\ref{sec:samerates}, we consider $\BB_t$ given by~(\ref{eq:BB}), and show that $\Pb(\BB_t) \geq \epsilon$ for some $\epsilon>0$, for all $t$ sufficiently large.
From that, one obtains site recurrence by the same reasoning as in Section~\ref{ss:siterecurrencezd}.

Again, choose $m=\frac{\delta^2}{Kt}$ as above.
By~(\ref{eq:avoideachotherDAneDB}) we have
\[
\tilde{\Pb} \big[X^{z,j} \meet X^{y,1} \text{ for some } z \in\A_m \setminus y \text{ and } j \big]\leq 2\delta^2
\]
for $t$ large enough, and an analogous argument gives
\begin{multline*}
\tilde{\Pb}
\big[X^{z,j}_+ \meet \oo \text{ for some } z \in\A_m \setminus y \text{ and } j \big]
\\
\leq
mK
\, \E
\left[ \text{\,number of sites visited by $X^{\oo,1}_+$ up to $t$\,} \right]
\leq
mK (t+1) \leq 2 \delta^2
.
\end{multline*}

We say that $X^{y,1}_-$ is \emph{good} if
\(
\tilde{\Pb} \left[ X^{z,j} \avoid X^{y,1} \ \forall z\in \A_m \setminus y \text{ and } j
\, \big| \,
X^{y,1}_-
\right] \geq 1-\delta
.
\)
Notice that by~(\ref{eq:avoideachotherDAneDB}) we have
\(
\Pb [X^{y,1}_- \mbox{ is good} ]
=
\tilde{\Pb} [X^{y,1}_- \mbox{ is good} ]
\geq 1-2\delta.
\)
The probability of~$\BB_t$ is estimated as follows.
\begin{align}
\Pb(\BB_t)
\geq
& \,
m
\textstyle
\sum_y
\tilde{\Pb}
\big[ X^{y,1}_- \meet \oo ,\, X^{y,1}_- \text{ good},\,
X^{z,j}_+ \avoid \oo,\, X^{z,j} \avoid X^{y,1} \text{ for all } z\in\A_m \setminus y \text{ and } j \big]
\nonumber
\\
=
& \,
m
\textstyle
\sum_y
\Pb
\big[ X^{y,1}_- \meet \oo ,\, X^{y,1}_- \text{ good} \big]
\times
\nonumber
\\
&
\times
\tilde{\E}
\left[ \tilde{\Pb} \left[
\begin{array}{c}
X^{z,j}_+ \avoid \oo \text{ and } X^{z,j} \avoid X^{y,1} \\ \text{ for all } z\in\A_m\setminus y \text{ and } j
\end{array}
\bigg| \, X^{y,1}_-
\right]
\,\Bigg|\, X^{y,1}_- \meet \oo ,\, X^{y,1}_- \text{ good} \right]
\nonumber
.
\end{align}
As in Section~\ref{sec:samerates}, the conditional probability is only being integrated on a subset of $[X^{y,1}_- \text{ good}]$, and $X^{y,1}_-$ is independent of $X^{z,j}_+$,
thus by a simple union bound we get
\begin{align}
\Pb(\BB_t)
&
\geq
m(1-\delta-2\delta^2)
\textstyle
\sum_y
\Pb
\big[ X^{y,1}_- \meet \oo ,\, X^{y,1}_- \text{ good} \big]
\nonumber
\\
&
=
m(1-\delta-2\delta^2)
\textstyle
\sum_y
\Pb
\big[ X^{\oo,1}_- \meet y ,\, X^{\oo,1}_- \text{ good} \big]
\nonumber
\\
&
=
m(1-\delta-2\delta^2)
\
\E
\big[ \# \text{ sites visited by }X^{\oo,1}_- \text{ during } [0,t],\, X^{\oo,1}_- \text{ good} \big]
\nonumber
.
\end{align}
As in Section~\ref{sec:samerates} this gives, for~$t$ large enough,
\begin{align*}
\Pb(\BB_t)
&
\geq
m (1-\delta-2\delta^2)
\, \gamma D_B \, t \, (1-2\delta-\delta)
=
\delta^2 \, \gamma D_B
\, (1-\delta-2\delta^2)
\, (1-3\delta)
K^{-1}
.
\end{align*}
Choosing $\delta=\frac{1}{4}$ and writing $\epsilon=\frac{D_B\gamma}{103 K}$, the above estimate yields $\Pb(\BB_t) \geq \epsilon$ for all~$t$ large enough.
As shown in Section~\ref{sec:samerates}, this implies site recurrence.

If the transition kernel~$p(\cdot,\cdot)$ yields a recurrent random walk on~$\Z^d$, the same argument given for $D_A=D_B$ in the previous section works.

\subsection{Fixed obstacles}
\label{sec:fixedobstacles}

We now prove site recurrence for the case $D_B = 0$ and $\mua_0 \geq \mub_0$ (in the case $\mua_0<\mub_0$, there is a positive density of fixed $B$-particles which survive forever and the state of the origin is eventually constant, being empty or containing $B$-particles. Hence, in this case the system is not site-recurrent).
The proof consists on showing that a system which is not site recurrent necessarily satisfies $\mub_0 > \mua_0$.
The latter assertion follows from the two propositions below.

\begin{proposition}
\label{prop:originnevervisited}
If the two-type annihilating system with fixed $B$-particles is not site recurrent, then,
with positive probability, $\xi_t(\oo) = \xi_0(\oo) \leq -1$ for all $t\geq 0$.
\end{proposition}

\begin{proposition}
\label{prop:annihilationofAparticles}
If, with positive probability, $\xi_t(\oo) = \xi_0(\oo) \leq -1$ for all $t\geq 0$, then, almost surely, every $A$-particle is eventually annihilated.
\end{proposition}

Proposition~\ref{prop:originnevervisited} implies that the density of $B$-particles does not vanish, that is, $\lim_t \mub_t > 0$.
On the other hand, by Proposition~\ref{prop:annihilationofAparticles}, every $A$-particle is eventually annihilated.
Therefore, using Lemma~\ref{lem:mtp}, we get that the density of $A$~particles satisfies $\lim_t \mua_t = 0$.
Again by Lemma~\ref{lem:mtp}, $\mua_t-\mub_t$ is constant in time, and therefore $\mua_0-\mub_0<0$.

\begin{proof}
[Proof of Proposition~\ref{prop:originnevervisited}]
Suppose that the system is not site recurrent, i.e.,
\begin{equation*}
\Pb[\boldsymbol o  \text{ is visited infinitely often}]<1.
\end{equation*}
This implies that, with positive probability, the set of $A$-particles which visit $\boldsymbol o$ is finite.
Thus, there exist $r\in\N$ and $x_1,\dots,x_r\in \Z^d$, such that $\Pb(\AA)>0$, where $\AA$ denotes the event that only particles starting at $x_1,\dots,x_r$ visit $\oo$ and these sites contain no $B$-particles.

Consider a pair of two-type annihilating systems $(\xi,\xi')$ constructed as follows.
We sample the same instructions for $\xi$ and $\xi'$.
Let $\xi_0$ be sampled as an i.i.d.~field with marginal~$\nu$.
Take $\xi'_0(x)=\xi_0(x)$ for $x\not\in\{\oo,x_1,\dots,x_r\}$.
For $x\in\{\oo,x_1,\dots,x_r\}$, sample $\xi'_0(x)$ independently with marginal~$\nu$.
Let us consider the events
\[
\big[\AA \text{ occurs for } \xi'\big] \text{ and }\big[\xi_0(x_i) \leq 0 \text{ for } i=1,\dots,r, \text{ and } \xi_0(\oo) \leq -1\big]
.
\]
Suppose that both of the above events occur.
Then all $A$-particles which visit the origin in~$\xi'$ are absent in~$\xi_0$.
Recalling that the systems share the same instructions, by Lemma~\ref{lem:monotononicity} the lifetime of other $A$-particles can only decrease compared to~$\xi'$, and therefore no $A$-particle can ever visit~$\oo$ in the system~$\xi$.
Since $\xi_0(\oo) \leq -1$, in the system~$\xi$ the site~$\oo$ does contain $B$-particles and they are never annihilated.

On the other hand, the events $\big[\AA \text{ occurs for } \xi'\big]$ and $\big[\xi_0(x_i) \leq 0 \text{ for } i=1,\dots,r, \text{ and } \xi_0(\oo) \leq -1\big]$ are independent, and they both have positive probability.
Hence the probability that $\xi_t(\oo) = \xi_0(\oo) \leq -1 \) for all $t$ is positive.
\end{proof}

\begin{proof}
[Proof of Proposition~\ref{prop:annihilationofAparticles}]
Since the law of the system is invariant under translations and under permutations of the labels of particles initially present at the same site, it suffices to show that,
almost surely on the event $[\xi_0(\boldsymbol o)\geq1]$, particle $(\oo,1)$ is eventually annihilated.

Consider a pair $(\xi,\xi')$ of two-type annihilating systems constructed as follows.
The initial condition is the same for $\xi$ and $\xi'$ and the instructions are the same except for the trajectory assigned to the first $A$-particle $(\oo,1)$ possibly present at $\oo$, which is chosen independently for~$\xi$ and~$\xi'$.

Define
\begin{equation*}
\B = \{x\in \Z^d: \xi'_t(x) = \xi'_0(x) \leq-1, \text{ for all } t\geq0\},
\end{equation*}
that is, $\B$ is the set of sites that initially contain at least one $B$-particle and that are never visited by $A$-particles in the system $\xi'$.
Since $\B$ is a translation co-variant function of initial conditions and instructions, which, in turn, are distributed as a product measure, it follows that $\B$ is ergodic under every translation on $\Z^d$.
By assumption, the set $\B$ has positive density.

Note that, on the event $[\xi'_0(\boldsymbol o)\geq 1 ]$, the system $\xi$ is obtained from $\xi'$ by deleting the first $A$-particle at $\oo$ and then placing a new one, with an independent trajectory, which we denote by $S^{\boldsymbol o,1}$.
By Lemma~\ref{lem:monotononicity}, this deletion cannot cause sites in $\B$ to be visited.
On the other hand, the walk $S^{\boldsymbol o,1}$ is independent of $\B$.
Therefore, by Lemma~\ref{lem:hittinganergodicset} below, $S^{\boldsymbol o,1}$ must hit the set $\B$ at some time $t$ at some site $x\in\B$.
Since~$x$ contains at least one $B$-particle and is never visited by other particles rather than $(\oo,1)$, the first $A$-particle at~$\oo$ in the system~$\xi$ is either annihilated before time~$t$, or it is annihilated at~$x$ at time~$t$.
\end{proof}

\begin{lemma}
\label{lem:hittinganergodicset}
Let $p(\cdot,\cdot)$ be a transition kernel on $\Z^d$ satisfying $p(x,x+y)=p(\oo,y)=p(\oo,-y)$.
Let $\B\subset\Z^d$ be a random set, whose distribution is ergodic and invariant with respect to translations.
Let $(X_n)_{n\in\N_0}$ be a random walk on $\Z^d$ which starts at $\boldsymbol o$ and jumps according to $p(\cdot,\cdot)$, and independent of $\B$.
Then almost surely $X$ hits $\B$ infinitely often.
\end{lemma}
\begin{proof}
Postponed to Appendix~\ref{sec:appendix}.
\end{proof}

%% file: generalgraphsab.tex
\section{Generously transitive graphs}
\label{sec:generalgraphs}

Let $G$ be a transitive, connected graph of finite degree, and let~$\oo$ denote an arbitrary site of~$G$.
We say that $G$ is \emph{generously transitive}\footnotemark\ if there is a group of automorphisms $\Gamma$ of~$G$ such that, for all $x,y\in G$ there exists $\pi\in\Gamma$ satisfying $\pi x = y$ and $\pi y = x$.

\footnotetext{%
A graph~$G$ being generously transitive is stronger than being unimodular, and it is neither stronger nor weaker than being Cayley.
An example of a graph that is generously transitive but not Cayley is the product $P \times \Z$, where $P$ is the Petersen graph. Cayley graphs of Abelian groups are generously transitive.
An example of a graph that is not generously transitive but is Cayley is the free product $\Z_2 * \Z_3$.
}

Let $p:G \times G \to [0,1]$ be a transition kernel.
Take~$\Gamma_p$ as the set of automorphisms~$\pi$ of~$G$ such that $p(\pi x, \pi y)=p(x,y)$ for all $x,y\in G$.
We say that $p$ is \emph{reflectable} if the group $\Gamma_p$ makes $G$ a generously transitive graph.

Generously transitive graphs are for instance, regular trees, finite complete graphs, and products of these, such as slabs with periodic boundary conditions.
Examples of reflectable walks include the uniform nearest-neighbor walk, or any walk whose transition probability depends only on the distance.

To avoid degenerate cases we assume that the sets
\begin{equation}
\label{eq:cx}
\C_x = \{w\in G:\exists n\in\N_0,\, p^n(x,w)>0\}
\end{equation}
are infinite.
In this section we prove the following:
\begin{theorem}
\label{thm:generalgraphs}
Let $(\xi_t)_{t \geq 0}$ be a two-type annihilating particle system on a generously transitive graph~$G$.
Suppose that the initial condition $\xi_0\in\Z^{G}$ is an i.i.d.~field whose marginal~$\nu$ on~$\Z$ is non-degenerate and has finite first moment.
Suppose that the jump rates are $0 \leq D_B \leq D_A = 1$ and that the jump distribution~$p(\cdot,\cdot)$ is reflectable.
Then there exists a universal $c>0$ such that
\[
\rho_t \geq \frac{c }{t}
\]
for all large enough $t$.
Moreover, unless $D_B = 0$ and $\mub_0>\mua_0$,
\[
\Pb\left[ \big. \xi_t(\oo) \text{ eventually constant} \right] = 0
,
\]
i.e., the system $\xi$ is site recurrent.
\end{theorem}

We will not present a self-contained proof.
Assuming that the reader has gone through the previous sections, we will focus on the parts of the proofs where the structure of~$\Z^d$ was used, and replace them accordingly.
The most delicate part, which we will do in detail, is the construction of a pair of entangled tracers, and verifying conditions for these tracers to meet in finite time.

By assumption, the distribution of~$(\xi_0,S,h)$ is $\Gamma_p$-invariant.
Moreover, the construction of the system from $(\xi_0,S,h)$ is $\Gamma_p$-covariant.
Since~$\Gamma_p$ makes~$G$ generously transitive, it is a unimodular group, and thus proofs based on mass-transport principle remain valid.
Since by assumption $\C_x$ is infinite, proofs based on ergodicity remain valid as well.
In particular, Lemmas~\ref{lem:existence}--\ref{lem:hittinganergodicset} hold in this setting.
See Appendix~\ref{sec:appendix}.

The coupling $(\xi,\xi^m)$ described in Sections~\ref{sec:samerates} and~\ref{sec:danedb} can be defined in the present setting.
Again, the difference between $\xi$ and $\xi^m$ is given by a family of tracers $(X^{x,j})_{x\in G, j=1,\dots,K}$, and relation~(\ref{eq:diffmanytracers}) holds.

In the proofs of Theorems~\ref{thm:zd},~\ref{thm:bl}, and~\ref{thm:DAneDB}, there are a few passages where a sum over~$\Z^d$ is rewritten, such as~(\ref{eq:switchtricktwopaths}) and~(\ref{eq:switchtrickonepath}).
The desired identity follows from re-indexing the sum or, alternatively, by keeping the same indexes and using invariance under reflections.
In these places we can keep the indexes and consider, for each term in the sum, an automorphism $\pi \in \Gamma_p$ which swaps $\oo$ for $y$, $z$, or $z-y$.

The weak law of large numbers for the range of the walk used to obtain~(\ref{eq:visitpositive})
holds for transient random walks on generously transitive graphs whose transition kernel is reflectable.
This follows from the argument presented in~\cite[\S 6.2.1]{Hughes95}.

\bigskip

The step that has no immediate analogue is
$[X^y\meet X^z] = [X^z-X^y\meet\oo]$,
used in~(\ref{eq:differenceoftracerequalrates}) and~(\ref{eq:differenceofwalks}), together with the fact that $X^z-X^y$ is a process that jumps according to  $p(\cdot,\cdot)$.
For a general graph~$G$ the subtraction $X^z-X^y$ is not even defined.

We want a representation of $X^y$ and $X^z$ that provides a treatable characterization of the event $[X^z\meet X^y]$.
The construction below provides a process $Y^{yz}=(Y^{yz}_t)_{t \geq 0}$ which jumps according to $p(\cdot,\cdot)$, and with the property that $d(X^z_t,X^y_t)=d(Y^{yz}_t,z)$.
In particular, $[X^z\meet X^y] = [Y^{yz}\meet z]$.

The main step is to find a coupling at the discrete-time level.
Let $Z=(Z_n)_{n\in\N_0}$ be a discrete-time random walk on $G$ starting at $Z_0=y$ with transition kernel $p(\cdot,\cdot)$.
Let~$\ell=(\ell_1,\ell_2,\ell_3,\dots)\in\{1,2\}^\N$ and take $\ell_0=1$.

We will construct a pair of processes $(W^y_n)_{n\in\N_0}$ and $(W^z_n)_{n\in\N_0}$ with $W^y_0=y$ and $W^z_0=z$, satisfying
the following properties.
First, $d(W^y_n,W_n^z)=d(Z^y_n,z)$ for all ${n\in\N_0}$.
Second,
the conditional distribution of $(W^y_{n+1}, W^z_{n+1})$ given $(W^y_i,W^z_i)_{i\leq n}$ and $(\ell_i)_{i \leq n+1}$ is given by $p(W^y_n,\cdot) \otimes \delta_{W^z_n}$ if $\ell_{n+1}=1$ and $\delta_{W^y_n} \otimes p(W^z_n,\cdot)$ if $\ell_{n+1}=2$.
The role of $\ell_n$ here is to indicate which of the walks is going to jump.

Let us describe the construction.
For each $x,w\in G$, fix some $\pi^{x,w}\in\Gamma_p$ such that $\pi^{x,w} x = w$ and $\pi^{x,w} w = x$.
At step~$n=0$ we take
\[
Z^0=Z
, \quad
z_0=z
, \quad
W^y_0=Z^0_0=y
, \quad
W^z_0=z_0=z
.
\]
For $n \in \N$, take
\[
\pi^n =
\begin{cases}
\mathrm{Id}, & \ell_n=\ell_{n-1}, \\
\pi^{Z^{n-1}_{n-1},z_{n-1}}, & \ell_{n} \ne \ell_{n-1}.
\end{cases}
\]
Take
\[
Z^n = \pi^n Z^{n-1}
\quad \text{ and } \quad
z_n = \pi^n z_{n-1}
\]
and
\[
W^y_n =
\begin{cases}
Z^n_n ,& \ell_n=1, \\
z_n ,& \ell_n=2,
\end{cases}
\qquad
\qquad
W^z_n =
\begin{cases}
z_n ,& \ell_n=1, \\
Z^n_n ,& \ell_n=2.
\end{cases}
\]

The first property is immediate.
Indeed, writing $\pi^{n!}=\pi^n\cdots\pi^1$, we have
\[
d(W^y_n,W^z_n) = d(Z^n_n,z_n) = d( \pi^{n!} Z^0_n,\pi^{n!} z_0) = d(Z_n,z).
\]

For the second property, assume that $\ell_{n+1}=1$.
The case $\ell_{n+1}=2$ is analogous.

If $\ell_n=1$, it means that~$\pi^{n+1}$ is the identity and $W^z_{n+1}=z_{n+1}=z_n=W^z_{n}$.
Moreover, $W^y_n=Z^n_n$, and $W^y_{n+1}=Z^{n+1}_{n+1}=Z^{n}_{n+1}$.
Now $Z^n=\pi^{n!} Z$, and the conditional distribution of $Z_{n+1}$ given $Z_1,\dots,Z_n$, is $p(Z_n,\cdot)$.
Since $\pi^1,\dots,\pi^{n}\in\Gamma_p$, the conditional distribution of $Z^{n}_{n+1}=\pi^{n!}Z_{n+1}$ given $Z_1,\dots,Z_n$ and $\ell_1,\dots,\ell_n$ is given by $p(\pi^{n!}Z_{n},\cdot)$, which in turn equals $p(W^y_{n},\cdot)$.

If $\ell_n=2$, it means that $\pi^{n+1}=\pi^{Z^n_n,z_n}$ and $W^z_{n+1}=z_{n+1}=\pi^{Z^n_n,z_n}z_n=Z^n_n=W^z_n$.
Moreover,
$W^y_n=z_n=\pi^{Z^n_n,z_n}Z^{n}_{n}=\pi^{Z^n_n,z_n}\pi^{n!}Z_{n}$
and
$W^y_{n+1}=Z^{n+1}_{n+1}=\pi^{Z^n_n,z_n}Z^{n}_{n+1}=\pi^{Z^n_n,z_n}\pi^{n!}Z_{n+1}$.
As in the previous case, the conditional distribution of $Z_{n+1}$ given $Z_1,\dots,Z_n$ is $p(Z_n,\cdot)$.
Again, $\pi^{Z^n_n,z_n}\pi^{n!} \in \Gamma_p$, and thus
the conditional distribution of $Z^{n}_{n+1}=\pi^{Z^n_n,z_n}\pi^{n!}Z_{n+1}$ given $Z_1,\dots,Z_n$ and $\ell_1,\dots,\ell_n$ is given by $p(\pi^{Z^n_n,z_n}\pi^{n!}Z_{n},\cdot)$, which in turn equals $p(W^y_{n},\cdot)$.

We finally describe the continuous-time construction using the above one.
This is the last missing step for Theorem~\ref{thm:generalgraphs} to be proved along the same lines as Theorem~\ref{thm:DAneDB}.

Let $y$ and $z$ be fixed.
Sample a quintuple $\Y^{yz} = (Z,\T^A_y,\T^B_y,\T^A_z,\T^B_z)$, where $Z=(Z_n)_{n\in\N_0}$ is a random walk starting at~$Z_0=y$ and jumping according to $p(\cdot,\cdot)$, and the clocks are given as in Section~\ref{sec:danedb}.
As before, the entangled tracers~$X^{y,1}$ and $X^{z,1}$ will be constructed from this quintuple.

The sequences~$(\ell_n)_{n\in\N}$, $(W^y_n)_{n\in\N}$, and $(W^z_n)_{n\in\N}$ will be defined dynamically.
Starting with $n=0$, define $\ell_0$, $W^y_0$, and $W^z_0$ as above.
Let the tracers $X^{y,1}$ and $X^{z,1}$ start at positions~$y$ and~$z$ and listen to the appropriate clock, as in Section~\ref{sec:danedb}.
When one of these tracers is supposed to jump due to a clock ring, we increment the value of~$n$, and take~$\ell_n$ as~$1$ or~$2$ depending on whether~$X^{y,1}$ or~$X^{z,1}$ is going to jump.
Knowing the value of~$\ell_n$ we can define $W^y_n$ and $W^z_n$, which will be the new positions of~$X^{y,1}$ and~$X^{z,1}$.
Carrying this procedure indefinitely, we obtain a sequence~$\ell\in\{1,2\}^{\N}$

As in Section~\ref{sec:danedb}, we define the path~$Y^{yz}_+ = \big( Y^{yz}_+(t) \big)_{t \geq 0}$ by following the discrete path~$Z^{yz}$ and jumping on $\T^A_{y} \cup \T^A_{z}$, that is, it jumps when any of the clocks ring.
So again we have
\[ \left[ Y^{yz} \meet w \right] \subseteq \left[ Y^{yz}_+ \meet w \right]. \]
This finishes the construction of the entangled tracers.

%% file: appendix.tex
\section{Postponed proofs}
\label{sec:appendix}

As mentioned at the beginning of Section~\ref{sec:construction}, the construction
was described for~$\Z^d$
with $p(x,x+y)=p(\oo,y)$, but works for any graph~$G$ having a transitive unimodular group of automorphisms~$\Gamma$ such that $p(x,y)=p(\pi x,\pi y)$ for any $\pi\in\Gamma$ and $x,y\in G$.
We present the proofs for this setting.

During this appendix we will make use of the \emph{mass transport principle}, which we now briefly recall, referring to~\cite[\S 8]{lyons-peres-} for details.
We say that $f:G\times G\to \R$ is \emph{diagonally invariant under $\Gamma$} if
\(f(x,y)=f(\pi x,\pi y)\) for all $x\in G$ and $\pi \in \Gamma$.
Under our assumption that $\Gamma$ is unimodular, we can apply~\cite[Corollary~8.8]{lyons-peres-} which says that 
\begin{equation}
\label{eq:masstransport}
\sum_{y\in G} f(x,y)=\sum_{y\in G} f(y,x)
\end{equation}
 for all $x\in G$.
Later on we will assume that $\Gamma_p$ makes $G$ generously transitive, which implies that it is unimodular so

\begin{lemma*}
[Lemma~\ref{lem:existence}\ restated]
Under the above assumptions, the construction described before Lemma~\ref{lem:existence} is well-defined and is $\Gamma_p$-covariant.
\end{lemma*}

\begin{proof}
Let $\B(w,n):=\{y\in G: d(w,y)\leq n\}$.
Take ${\xi}^{n}_0=\xi_0\cdot\I_{\B(\oo,n)}$, that is, the initial condition $\xi_0$ with all particles outside $\B(\oo,n)$ deleted.
Consider the truncated system given by $({\xi}^{n}_0,S,h)$.
The truncated system is well-defined, since it contains finitely many particles.

Define $\T_{y,j}^n$ as the time of annihilation of the particle $(y,j)$ in this system.
We set $\T_{y,j}^n=0$ if the particle $(y,j)$ is initially absent and $\T_{y,j}^n=\infty$ if the particle survives forever.
Notice that the history of particle $(y,j)$ can be reconstructed from~$S^{y,j}$ and~$\T_{y,j}^n$.
We will show that, almost surely, $\T_{y,j}^n=\T_{y,j}^m$ for all $m$ and $n$ large enough.
As a consequence, the construction of the full system can be defined as the limit of truncated systems as $n\to\infty$.

Recall from Section~\ref{sec:construction} that the difference between two systems which share the same instructions but have different initial conditions can be followed by a set of \emph{tracers}.
Let $(X^{n,x,i})_{x,i}$ denote the set of tracers that keep track of the differences between $\xi^{n}$ and $\xi^{n+1}$,
where~$x$ ranges over~$G$ and~$i$ ranges over $\{1,\dots,|\xi^{n+1}_0(x)-\xi^n_0(x)|\}$.
Denote by~$R_T(X^{n,x,i})$ the set of sites visited by $X^{n,x,i}_t$ during $t\in[0,T]$.
Now notice that $\T^n_{y,j}\wedge T$ may differ from $\T^{n+1}_{y,j}\wedge T$ only if some of these tracers intersects $S^{y,j}$ before time $T$.
Therefore we have, for any $L>0$,
\begin{multline}
\label{eq:upperboundTdiffer}
\Pb \big[\T^n_{y,j} \wedge T \ne \T^m_{y,j}\wedge T \text{ for infinitely many } (m,n) \big]
\leq
\Pb[R_T(S^{y,j})\not\subseteq \B(\oo,L)]
\,+
\\
+
\Pb \big[\text{infinitely many tracers } X^{n,x,i} \text{ visit } \B(\oo,L) \text{ by time $T$\,}\big].
\end{multline}

Let $X^x$ denote a random walk of jump rate~$D_A$ starting at $x$ and jumping according to $p(\cdot,\cdot)$ and $R_T(X^x)$ its range up to time $T$.
For all $L\in\N$
\begin{align}
\nonumber
\sum_{n\in\N}\Pb\big[\exists x & \in\partial^{\scriptscriptstyle+}\B(\oo,n),i\leq|\xi_0(x)|: R_T(X^{n,x,i})\cap \B(\oo,L)\neq \emptyset\big]\\
\nonumber
&\leq \sum_{n\in\N}\sum_{x\in\partial^{\scriptscriptstyle+}\B(\oo,n)} \sum_{i\geq 1} \nonumber
\Pb\big[ |\xi_0(x)|\geq i\big]\Pb\big[R_T(X^{x})\cap \B(\oo,L)\neq \emptyset \big]\\
\nonumber
&=\sum_{x\in G}\E\big[|\xi_0(x)|\big] \, \Pb[R_T(X^{\oo})\cap \B(x,L)\neq \emptyset]\\
\label{eq:sumoverg}
&\leq\sup_{x\in G}\E\big[|\xi_0(x)|\big] \, \sum_{x\in G} \, \Pb[R_T(X^{\oo})\cap \B(x,L)\neq \emptyset]\\
\nonumber
&\leq\sup_{x\in G}\E\big[|\xi_0(x)|\big] \, \E[\#\{y\in G: d(R_T(X^\oo),y)\leq L\}]\\
\nonumber
&\leq\sup_{x\in G}\E\big[|\xi_0(x)|\big] \, |\B(\oo,L)| \ \E\big[|R_T(X^\oo)|\big]\\
\nonumber
&\leq\sup_{x\in G}\E\big[|\xi_0(x)|\big] \, |\B(\oo,L)| \ D_A T<\infty.
\end{align}
Hence, by virtue of the Borel-Cantelli lemma the last term in~(\ref{eq:upperboundTdiffer}) is zero.
On the other hand, $\Pb[R_t(S^{y,j})\not\subseteq B(\oo,L)]\to0$ as $L\to\infty$.
Therefore,
\[\Pb[\T^n_{y,j}\wedge T = \T^m_{y,j}\wedge T \text{ for } m \text{ and } n \text{ large enough}]=1.\]

We have shown that the full system can be defined as the limit of truncated systems~$(\xi^n_0,S,h)$ having null initial condition outside $\B(\oo,n)$.

We need to show that this construction is $\Gamma_p$-covariant, where $\Gamma_p$ is the group of automorphisms defined in the begging of Section~\ref{sec:generalgraphs}.
We will show that the limit is the same if instead we take truncations on~$\B(w,n)$, for any fixed~$w\in G$.
Let $w\in G$ be fixed.
Take $\tilde{\xi}^{n}_0=\xi_0\cdot\I_{\B(w,n)}$, that is, the initial condition $\xi_0$ with all particles outside $\B(w,n)$ deleted.
Consider the truncated system given by $(\tilde{\xi}^{n}_0,S,h)$.
By the above argument, the limit of these truncated systems is well-defined.

Let $y \in G$ and $j\in\Z^*$.
We want to show that $\Pb[\tilde{\T}_{y,j} = \T_{y,j}] = 1$.
It is enough to show that $\Pb[\T^n_{y,j} \wedge T \neq \tilde{\T}^n_{y,j} \wedge T] \to 0$ as $n\to\infty$, for any fixed $T>0$.

Let $({X}^{n,x,i})_{x,i}$ be the set of tracers which keep track of the differences between $\xi^n$ and $\tilde{\xi}^n$, where $x$ ranges over $\B(\oo,n)\triangle \B(w,n)$ and $1 \leq i \leq|\xi_0(x)|$.
As before, $\T^n_{y,j}\wedge T$ may differ from $\tilde{\T}^n_{y,j} \wedge T$ only if one of such tracers intersects $S^{y,j}$ before time~$T$.
Therefore, for any $L>0$,
\begin{multline}
\Pb[\T^n_{y,j}\wedge T\neq \tilde{\T}^n_{y,j}\wedge T]
\leq
\\
\leq
\Pb[R_{T}(S^{y,j})\not\subseteq \B(\oo,L)]
+
\sum_{x}
\E\big[|\xi_0(x)| \big]
\Pb[R_{T}({X}^{x})\cap \B(\oo,L)\neq\emptyset]
,
\nonumber
\end{multline}
where the sum is over $x\in \B(\oo,n)\triangle \B(w,n)$.
As in the previous argument, the last term is bounded by
\begin{multline*}
\sup_{x\in G} \E\big[|\xi_0(x)| \big]
\sum_{\qquad\mathclap{\scriptscriptstyle \B(\oo,n)\triangle \B(w,n)}\quad}
\Pb[R_{T}({X}^{\oo})\cap \B(x,L)\neq\emptyset]
\leq
\\
\leq
\sup_{x\in G} \E\big[|\xi_0(x)| \big]
\sum_{\qquad\mathclap{\scriptscriptstyle\B(\oo,n-d(\oo,w))^c}\quad}
\Pb[R_{T}({X}^{\oo})\cap \B(x,L)\neq\emptyset]
.
\end{multline*}
Since~(\ref{eq:sumoverg}) is finite and
$\B(\oo,n-d(\oo,w))^c {\to}\emptyset$, the above quantity vanishes as $n\to\infty$.
Finally,
\[
 \limsup_n
\Pb[\T^n_{y,j}\wedge T\neq \tilde{\T}^n_{y,j}\wedge T]
\leq
\Pb[R_{T}(S^{y,j})\not\subseteq \B(\oo,L)]
\to
0
\text{ as }
L \to \infty,
\]
finishing the proof.
\end{proof}

\begin{proof}
[Proof of Lemma~\ref{lem:mtp}]
Let 
\(f_t(x,y)=\sum_{i,j\in\N,z\in G} \Pb[{M(x,i,z,-j,y,t)}] \)
denote the expected number of particles of type $A$ which started at $x$ and which have been annihilated at $y$ up to time $t$. 
Then
\(\sum_{x\in G} f_t(x,\oo)\)
equals the expected number of particles of type $A$ which have been annihilated at $\oo$ up to time $t$. This, together with the fact that each annihilation at $y$ involves one particle of type $A$, yields
\begin{equation}
\label{eq:quica}
\sum_{x\in G} f_t(x,\oo)=\Theta_t.
\end{equation} 
On the other hand,
\(\sum_{y\in G} f_t(x,y)\)
equals the expected number of particles of type $A$ started at $x$ which have been annihilated up to time $t$, and therefore
\begin{equation}
\label{eq:popeye}
\mu_0^A-\mu_t^A=\sum_{y\in G} f_t(\oo,y).
\end{equation}
Since the jump distribution is invariant under $\Gamma_p$, $f_t$ is diagonally invariant,
we can apply~(\ref{eq:masstransport}) to~(\ref{eq:quica}) and~(\ref{eq:popeye}) to obtain
\(
-\Theta_t=\mu^A_t-\mu^A_0.
\)
An analogous reasoning gives that
\(
-\Theta_t=\mu^B_t-\mu^B_0.
\)
This proves the first part of the lemma.

To prove the second claim, let
$\tilde{V}(x,j,y,t)$ denote the event that particle $(x,j)$ is present at site $y$ at time $t$.
Let $g_t(x,y)=\sum_{j\in\N}\Pb[\tilde{V}(x,j,y,t)]$ denote the expected number of $A$-particles which started at $x$ and are present at $y$ at time $t$.
Then
\(\mu^A_t=\sum_{x\in G} g_t(x,\oo)\)
and
\(\sum_{j\in\N}\Pb[V(\oo,j,t)]=\sum_{y\in G} g_t(\oo,y).\)
By arguments analogous to those above, $g_t$ is diagonally invariant under $\Gamma_p$, and~(\ref{eq:masstransport}) yields the second part of the lemma.
\end{proof}

We do not restate Lemma~\ref{lem:monotononicity}, it suffices to replace $\Z^d$ by 
$G$.

\begin{proof}
[Proof of Lemma~\ref{lem:monotononicity}]
For systems with finitely many particles, these monotonicity properties are obvious.
On the other hand, from the proof of Lemma~\ref{lem:existence}, the real system~$\xi$, as well as the variables~$\T_{x,j}$, can be approximated by large finite systems.
The lemma follows from these two observations.
\end{proof}



For the proof of Lemma~\ref{lem:zeroone},
recall the definition of~$\C_x$ from~(\ref{eq:cx}).
Since we are now assuming that $p(\cdot,\cdot)$ is reflectable, these sets split $G$ into equivalence classes.
Moreover, $\pi \C_x = \C_{\pi x}$ for any $\pi\in\Gamma_p$ and $x\in G$.
Take
\[ \Gamma'=\{ \pi\in\Gamma_p : \pi\oo \in \C_\oo\}. \]
Since the orbit of $\oo$ under $\Gamma'$ is $\C_\oo$, which we are assuming to be infinite, and the graphical construction is a $\Gamma'$-covariant function of $(\xi_0,S,h)$, whose distribution is a $\Gamma'$-invariant product measure, it follows that the construction is $\Gamma'$-ergodic.
See section ``Tolerance and Ergodicity'' in~\cite{lyons-peres-}.

\begin{proof}
[Proof of Lemma~\ref{lem:zeroone}]
It suffices to show that $\CC_\oo$ has probability~$0$ or~$1$, where $\CC_x$ denotes the event $[\text{site $x$ is visited by $A$-particles infinitely many times}]$.
The same proof works for $B$-particles in case $D_B>0$.

Suppose that the random walk with transition kernel $p(\cdot,\cdot)$ is transient.
For finite sets $\emptyset = \B_0 \subseteq \B_1 \subseteq \cdots$ with $\B_n \uparrow G$, let $\xi^n_0 = \xi_0 \cdot \I_{\B_n^c}$.
Write $\CC_\oo^n$ for the occurrence of~$\CC_\oo$ in the system given by~$(\xi_0^n,S,h)$, so that $\CC_\oo^0 = \CC_\oo$.
Notice that the system given by $(\xi^n_0,S,h)$ depends only on $(\xi_0(x),S^{x,\cdot},h^{x,\cdot})_{x\not\in \B_n}$, because the particles starting in~$\B_n$ are deleted.
Therefore $\limsup_n \CC_\oo^n$ is a tail event, and it suffices to show that
\[\Pb(\CC_\oo^n \triangle \CC_\oo^{n+1}) = 0 \text{ for all } n.\]

To see why this is true, we look at the difference between $\xi^n_0$ and $\xi^{n+1}_0$.
Similarly as in Sections~\ref{sec:construction},~\ref{sec:danedb}, and~\ref{sec:generalgraphs}, we consider a set of tracers $(X^{x,i})_{x,i}$,
where $x$ ranges over $G$ and $i$ ranges over $\{1,\dots,|\xi^{n+1}_0(x)-\xi^n_0(x)|\}$, and such tat
\[
\xi^{n+1}_t - \xi^{n}_t
=
\sum_{x} \sum_{i} \sgn[\xi^{n+1}_0(x) - \xi^{n}_0(x)] \cdot \I_{[X^{x,i} \text{ active at time }t]} \cdot \delta_{X^y_t}
.
\]
Now for the event $\CC_\oo^n \triangle \CC_\oo^{n+1}$ to hold, necessarily site~$\oo$ is visited by these tracers infinitely often.
But the tracers jump according the transition kernel $p(\cdot,\cdot)$, which we are assuming to be transient and, since there are finitely many such tracers, we deduce that $\Pb(\CC_\oo^n \triangle \CC_\oo^{n+1})=0$.

Now suppose that random walks are recurrent.
We claim that $\CC_\oo$ a.s.\ implies $\cap_{x\in\C_\oo}\CC_x$.
The converse implication is trivial.
Since the latter event is $\Gamma'$-invariant, it follows that its probability is either~$0$ or~$1$.

It remains to prove the claim.
Each time site~$\oo$ is visited by an $A$-particle, either that particle is annihilated at~$\oo$ or it jumps to a site chosen according to~$p(\oo,\cdot)$.
By Lemma~\ref{lem:mtp}, $\E[\mbox{number of annihilations at }\oo] \leq \mua_0 < \infty$.
Hence, infinitely many visits to $\oo$ almost surely imply infinitely many visits to any~$y$ such that $p(\oo,y)>0$.
By induction on~$n$, infinitely many visits to $\oo$ almost surely imply infinitely many visits to any~$y$ such that $p^n(\oo,y)>0$ for some $n\in\N_0$, proving the claim.
\end{proof}

\begin{lemma*}
[Lemma~\ref{lem:hittinganergodicset}\ restated]
If $p(\cdot,\cdot)$ is a reflectable transition kernel of~$G$, $\B\subseteq G$ is $\Gamma'$-ergodic and $X_n$ is a random walk which starts at~$\oo$ and jumps according to $p(\cdot,\cdot)$, independent of $\B$, then, almost surely, $X$ visits~$\B$ infinitely often.
\end{lemma*}

\begin{proof}
Write
\[
q_n = \sup_{k\in\N_0,w\in\B} p^k(X_n,w)
.
\]
Since the set~$\B$ is $\Gamma'$-ergodic with positive density, we have
\[ \Pb[q_0>0]=\Pb[\B\cap \C_\oo \ne \emptyset]=1. \]
Now notice that the sequence $(q_n)_{n\in\N}$ is stationary (because $X$~is independent of~$\B$), and thus $\Pb[q_n \to 0]=0$.
Hence, almost surely, there exists random $\delta>0$ such that $q_n \geq 2\delta$ infinitely often.
Therefore, for each time~$n$ when $q_n \geq 2\delta$, there is a random $k_n\in\N_0$ such that $\Pb \left[ X_{n+k}\in\B \,\big|\, X_1,\dots,X_n; \B \right] \geq \delta$.
The claim then follows from a conditional Borel-Cantelli lemma.
\end{proof}

%% file: abparticles.bbl
\begin{thebibliography}{10}

\bibitem{adelman-76}
O.~Adelman.
\newblock Some use of some ``symmetries'' of some random process.
\newblock {\em Ann. Inst. H. Poincar{\'e} Sect. B (N.S.)}, 12:193--197, 1976.

\bibitem{andjel-82}
E.~D. Andjel.
\newblock Invariant measures for the zero range processes.
\newblock {\em Ann. Probab.}, 10:525--547, 1982.

\bibitem{arratia-83}
R.~Arratia.
\newblock Site recurrence for annihilating random walks on {${\bf Z}\sb{d}$}.
\newblock {\em Ann. Probab.}, 11:706--713, 1983.

\bibitem{balagurov-vaks-73}
B.~V. Balagurov and V.~G. Vaks.
\newblock Random walks of a particle on lattices with traps.
\newblock {\em Zh. Eksp. Teor. Fiz.}, 65:1939--1946, 1973.

\bibitem{BenjaminiFoxallGurel-GurevichJungeKesten16}
I.~Benjamini, E.~Foxall, O.~Gurel-Gurevich, M.~Junge, and H.~Kesten.
\newblock Itai benjamini, eric foxall, ori gurel-gurevich, matthew junge, and
  harry kesten.
\newblock {\em Electron. Commun. Probab.}, 21:47, 2016.

\bibitem{bramson-lebowitz-88}
M.~Bramson and J.~L. Lebowitz.
\newblock Asymptotic behavior of densities in diffusion-dominated annihilation
  reactions.
\newblock {\em Phys. Rev. Lett.}, 61:2397--2400, 1988.

\bibitem{bramson-lebowitz-90}
M.~Bramson and J.~L. Lebowitz.
\newblock Asymptotic behavior of densities in diffusion dominated two-particle
  reactions.
\newblock {\em Phys. A}, 168:88--94, 1990.

\bibitem{bramson-lebowitz-91}
M.~Bramson and J.~L. Lebowitz.
\newblock Asymptotic behavior of densities for two-particle annihilating random
  walks.
\newblock {\em J. Statist. Phys.}, 62:297--372, 1991.

\bibitem{bramson-lebowitz-91b}
M.~Bramson and J.~L. Lebowitz.
\newblock Spatial structure in diffusion-limited two-particle reactions.
\newblock {\em J. Statist. Phys.}, 65:941--951, 1991.

\bibitem{bramson-lebowitz-01}
M.~Bramson and J.~L. Lebowitz.
\newblock Spatial structure in low dimensions for diffusion limited
  two-particle reactions.
\newblock {\em Ann. Appl. Probab.}, 11:121--181, 2001.

\bibitem{cabezas-rolla-sidoravicius-14}
M.~Cabezas, L.~T. Rolla, and V.~Sidoravicius.
\newblock Non-equilibrium phase transitions: Activated random walks at
  criticality.
\newblock {\em J. Stat. Phys.}, 155:1112--1125, 2014.

\bibitem{dickman-rolla-sidoravicius-10}
R.~Dickman, L.~T. Rolla, and V.~Sidoravicius.
\newblock Activated random walkers: Facts, conjectures and challenges.
\newblock {\em J. Stat. Phys.}, 138:126--142, 2010.

\bibitem{erdos-ney-74}
P.~Erd\H{o}s and P.~Ney.
\newblock Some problems on random intervals and annihilating particles.
\newblock {\em Ann. Probability}, 2:828--839, 1974.

\bibitem{griffeath-78}
D.~Griffeath.
\newblock Annihilating and coalescing random walks on {${\bf Z}\sb{d}$}.
\newblock {\em Z. Wahrsch. Verw. Gebiete}, 46:55--65, 1978.

\bibitem{griffeath-79}
D.~Griffeath.
\newblock {\em Additive and cancellative interacting particle systems}, volume
  724 of {\em Lecture Notes in Mathematics}.
\newblock Springer, 1979.

\bibitem{holley-stroock-79}
R.~Holley and D.~Stroock.
\newblock Dual processes and their application to infinite interacting systems.
\newblock {\em Adv. in Math.}, 32:149--174, 1979.

\bibitem{holley-liggett-75}
R.~A. Holley and T.~M. Liggett.
\newblock Ergodic theorems for weakly interacting infinite systems and the
  voter model.
\newblock {\em Ann. Probability}, 3:643--663, 1975.

\bibitem{Hughes95}
B.~D. Hughes.
\newblock {\em Random walks and random environments. {V}ol. 1}.
\newblock Oxford Science Publications. The Clarendon Press, Oxford University
  Press, New York, 1995.
\newblock Random walks.

\bibitem{koritskii-etal-60}
A.~Koritskii, I.~Molin, V.~Shamshev, N.~Buben, and V.~Voevodskii.
\newblock An electronic paramagnetic resonance study of the radicals formed in
  fast electron irradiation of polyethylene.
\newblock {\em Polymer Science U.S.S.R.}, 1:458--472, 1960.

\bibitem{lootgieter-77}
J.-C. Lootgieter.
\newblock Probl{\`e}mes de r{\'e}currence concernant des mouvements
  al{\'e}atoires de particules sur {${\bf Z}$} avec destruction.
\newblock {\em Ann. Inst. H. Poincar{\'e} Sect. B (N.S.)}, 13:127--139, 1977.

\bibitem{lyons-peres-}
R.~Lyons and Y.~Peres.
\newblock Probability on trees and networks.
\newblock Book in preparation. Current version available at
  \url{http://mypage.iu.edu/\string~rdlyons}.

\bibitem{ovchinnikov-zeldovich-78}
A.~Ovchinnikov and Y.~Zeldovich.
\newblock Role of density fluctuations in bimolecular reaction kinetics.
\newblock {\em Chem. Phys.}, 28:215--218, 1978.

\bibitem{ovchinnikov-belyi-68}
A.~A. Ovchinnikov and A.~A. Belyi.
\newblock The kinetics of the destruction of radicals in polymers.
\newblock {\em Theor. Exp. Chem.}, 2:405--408, 1968.

\bibitem{rolla-sidoravicius-12}
L.~T. Rolla and V.~Sidoravicius.
\newblock Absorbing-state phase transition for driven-dissipative stochastic
  dynamics on {$Z$}.
\newblock {\em Invent. Math.}, 188:127--150, 2012.
\newblock \href{http://arxiv.org/abs/0908.1152}{arXiv:0908.1152}.

\bibitem{schwartz-78}
D.~Schwartz.
\newblock On hitting probabilities for an annihilating particle model.
\newblock {\em Ann. Probability}, 6:398--403, 1978.

\bibitem{toussaint-wilczek-83}
D.~Toussaint and F.~Wilczek.
\newblock Particle-antiparticle annihilation in diffusive motion.
\newblock {\em J. Chem. Phys.}, 78:2642, 1983.

\end{thebibliography}
